\documentclass[12pt]{amsart}

\usepackage[margin=1in]{geometry}
\usepackage[utf8]{inputenc}
\usepackage[english]{babel}
\usepackage{amssymb}
\usepackage{amsmath}
\usepackage{amsthm}
\usepackage{dsfont}
\usepackage{amsmath}
\usepackage{cite}
\usepackage{xcolor}

\usepackage [english]{babel}
\usepackage [autostyle, english = american]{csquotes}
\MakeOuterQuote{"}

\usepackage{hyperref}
\hypersetup{
    colorlinks,
    linkcolor={blue!80!black},
    citecolor={blue!60!black},
    urlcolor={blue!80!black},
}
\newtheorem{thm}{Theorem}[section]
\newtheorem{cor}[thm]{Corollary}
\newtheorem{lem}[thm]{Lemma}
\newtheorem{prop}[thm]{Proposition}

\theoremstyle{definition}
\newtheorem{ex}[thm]{Example}

\theoremstyle{definition}
\newtheorem{defi}[thm]{Definition}
\newtheorem*{definition-non}{Definition}

\theoremstyle{definition}
\newtheorem{com}[thm]{Comment}

\theoremstyle{remark}
\newtheorem*{rmkx}{Remark}
\newenvironment{rmk}
  {\pushQED{\qed}\rmkx}
  {\popQED\endrmkx}

\newcommand*\interior[1]{\mathring{#1}} 

\newcommand{\spt}{\text{spt}}

\newcommand{\eps}{\varepsilon}
\newcommand{\emps}{\varnothing}

\newcommand{\rn}{\mathbb R^n}

\newcommand{\sn}{S^{n-1}}

\newcommand{\kn}{\mathcal K^n}
\newcommand{\po}{\mathcal P}
\newcommand{\kno}{\mathcal K^n_o}

\newcommand{\bla}{\raise.2ex\hbox{$\scriptstyle\pmb \langle$}}
\newcommand{\sbla}{\raise.1ex\hbox{$\scriptscriptstyle\pmb \langle$}}
\newcommand{\bra}{\raise.2ex\hbox{$\scriptstyle\pmb \rangle$}}
\newcommand{\sbra}{\raise.1ex\hbox{$\scriptscriptstyle\pmb \rangle$}}

\newcommand{\balpha}{\pmb{\alpha}}

\numberwithin{equation}{section}

\begin{document}

\title{The Discrete Gauss Image problem}

\author{Vadim Semenov}
\address{
Brown University, 182 George Street, Providence, MA 02912}
\curraddr{}
\email{vadim\_semenov@brown.edu}
%\thanks{
%The author was supported by ...}

%\subjclass[2010]{Primary }

\date{\today}

\dedicatory{}


\maketitle

%\begin{abstract}
%The Gauss Image problem is a generalization to the question originally posed by Aleksandrov, who studied the existence of the convex body with prescribed Aleksandrov's integral curvature. A special, simple discrete case of the Gauss Image Problem can be formulated as follows: given a finite set of directions in $\rn$ and the same number of unit vectors, does there exist a convex polytope in $\rn$ containing the origin in its interior with vertices at given directions such that each normal cone at the vertex contains exactly one of the given vectors in its interior?
%
%We pose a combinatorial problem, called the Assignment Problem, for discrete measures. It is shown that the discrete Gauss Image Problem and the Assignment Problem are equivalent. We establish a sufficient, almost necessary, geometric condition for measures, which solves the Assignment Problem and the Gauss Image Problem. The proper reformulation for the uniqueness question is also addressed and analyzed. The work establishes interesting connections of the Discrete Gauss Image problem to Hall's marriage theorem and transportation polytopes.
% \end{abstract}

 \begin{abstract}
We study the Discrete Gauss Image Problem, a generalization of Aleksandrov's classical question on the existence of convex bodies with prescribed integral curvature. We introduce a combinatorial problem called the Assignment Problem and show its equivalence to the Discrete Gauss Image Problem. We establish sufficient (and nearly necessary) geometric conditions on measures that solve both problems. Additionally, we provide new discrete interpretations of some classical concepts related to Aleksandrov's integral curtvature, such as, for example, connecting Aleksandrov relation to Hall's Marriage Theorem. \end{abstract}

\setcounter{tocdepth}{1}
\tableofcontents

\newpage

\section{Introduction}
\subsection{From Aleksandrov to Today}

The classical Aleksandrov problem is a natural and well--known counterpart to the Minkowski problem. Introduced by Aleksandrov in \cite{Aleks1}, it studies the existence and uniqueness of a convex body with prescribed Aleksandrov's integral curvature, the second most important area measure. When the convex body $K$ is sufficiently smooth, the Aleksandrov integral curvature of $K$ (when viewed as a measure on $\partial K$) has the Gauss curvature as its density. When the convex body $K$ is a polytope, it captures the sizes of the normal cones at the vertices, measuring their exterior angles. 

Using more analytic language, Aleksandrov \cite{Aleks0,Aleks1,Aleks} characterized all possible pullbacks of the spherical Lebesgue measure $\lambda$ from the "normal" sphere to the "radial" sphere using the \textbf{\textit{radial Gauss image map}} $\balpha_K$. This map is a composition of the multivalued Gauss map of the body $K$ (the usual Gauss map if the body is sufficiently smooth) and the radial map of $K$ (an inverse to a radial projection from the boundary of $K$ onto a fixed sphere with center located in the interior of $K$),  see \eqref{radial gauss map defenition}. Aleksandrov's approach \cite{Aleks0,Aleks1,Aleks}, as well as those of others, see \cite{Oliker}, involved solving an approximate version of the problem by considering a discrete approximation to one of the measures.  Even the more modern and general variations of the problem, see  \cite{Bertrand, GIP, Semenov}, seem unable to avoid various discrete approximations, which are hidden in the proofs behind the smooth assumptions on measures. 

%These observations suggest a potentially undiscovered discrete nature of the Aleksandrov Problem, offering possible new and natural explanations for some of the complicated smooth phenomena concerning Aleksandrov's integral curvature. This motivates the desire to investigate the Aleksandrov Problem by imposing the discrete assumption on the measure $\lambda$, bringing us to the following question:

Observations like these suggest an undiscovered discrete nature of the Aleksandrov Problem, potentially providing natural explanations for some complex smooth phenomena involving Aleksandrov's integral curvature. This prompts us to investigate the Aleksandrov Problem with a discrete measure  $\lambda$, leading to the following question:

\smallskip
\noindent{\textbf{The Discrete Gauss Image Problem}} \textit{Suppose $\lambda$ is a fixed discrete  measure on $\sn$. What are the necessary and sufficient conditions on the spherical measure $\mu$ in relation to $\lambda$, so that there exists a convex body $K$, containing the origin in its interior, such that $\mu$ is the pullback of $\lambda$ via the radial Gauss image map $\balpha_K$, that is}  
\begin{equation}\label{Aug.08.18.1}
  \mu(\cdot)=\lambda(\balpha_K(\cdot))?
\end{equation}

\textit{If such a body exists, to what extent is it unique?}
\smallskip
 
As a side note, the smooth version of equation \eqref{Aug.08.18.1} is associated with the following Monge-Amp\`ere type equation: 
 \begin{equation}\label{PDE}
g\Big(\frac{\nabla h +h \iota}{|\nabla h +h \iota|}\Big)
|\nabla h +h \iota|^{-n}
h\, \det\big(\nabla^2h + h I\big) = f,
\end{equation}
where $h:\sn\to(0,\infty)$ is the unknown function, $f$ and $g$ are, respectively, the densities of the measures $\mu$ and $\lambda$, the map $\iota : \sn \to \sn$ is the identity, $I$ is the standard Riemannian metric on $\sn$, and $\nabla h$
and $\nabla^2 h$ are, respectively, the gradient and the Hessian of $h$ with respect to $I$.

Apart from the direct connection to the works of Aleksandrov, the desire to investigate the transfer of spherical measures via the radial Gauss image map arises from the possibility of linking the classical Brunn-Minkowski theory with the recent and actively investigated dual Brunn-Minkowski theory \cite{HLYZ16}. The mentioned pullback, which we denote by $\lambda(K,\cdot)$, is known as the \textbf{\textit{Gauss image measure}}, see Section \ref{Preliminaries} and equation \eqref{Aug.19.121}. Specifically, when $\lambda$ is a spherical Lebesgue measure, $\lambda(K,\cdot)$ is mentioned Aleksandrov's integral curvature of the body $K$, see \cite{Aleks}. When $\lambda$ is Federer's $(n-1)^\text{th}$ curvature measure, $\lambda(K,\cdot)$ is the surface area measure of Aleksandrov-Fenchel-Jessen \cite{Aleks0}. Finally, the dual curvature measures, see \cite{HLYZ16}, are also the Gauss image measures. This positions the Discrete Gauss Image Problem as a crucial question for establishing the bridge between the Brunn-Minkowski and the dual Brunn-Minkowski theory. 
 
 To provide a vivid geometric illustration of the problem without referring to a measure theory, consider the following special case of the Discrete Gauss Image Problem, which can be derived from the original problem by assuming that $\mu$ and $\lambda$ have the same number of equal-weight atoms. See Section \ref{Preliminaries} for definitions.
 
\smallskip
\noindent{\textbf{Special Case of the Discrete Gauss Image Problem}} \textit{In $\mathbb R ^n$, suppose we are given two sets of unit vectors $\{v_1,\ldots, v_m\}$ and $\{u_1,\ldots, u_m\}$. Let $\mathcal P$ be the set of convex polytopes containing the origin in their interiors with vertices in the directions of $v_i$:
\begin{equation}\label{convex hull notation}
	\po=\{\text{\normalfont conv}\{\beta_iv_i\mid 1 \leq i \leq m\}\mid (\beta_1,\ldots,\beta_m)\in \mathbb R^m_{>0}\}.
\end{equation}
What are the necessary and sufficient conditions on the vectors $v_i$ and $u_j$ for the existence of a convex polytope $P\in\mathcal P$ such that every normal cone at each vertex of $P$ contains exactly one vector from the set $\{u_1,\ldots, u_m\}$ in its interior? That is, for the existence of $P\in \mathcal P$ and a permutation $\sigma\in S_{m}$ such that
\begin{equation}
	u_i\in \interior{N\big (P,r_P(v_{\sigma(i)})\big )},
\end{equation}
where $N\big (P,r_P(v_{\sigma(i)})\big )$ is the normal cone at the vertex $r_P(v_{\sigma(i)})$ of the polytope $P$.} 
\smallskip

More broadly, the discrete Gauss Image Problem is part of the flourishing and active developments of the Brunn-Minkowski and the dual Brunn-Minkowski theories, which have led to many great results both in convexity and beyond. Started by Firey \cite{Firey2}, and continued by Lutwak \cite{LutwakLp}, these theories have recently given rise to many interesting results and conjectures, as, for example, actively developing the log-Brunn-Minkowski inequality conjecture, see \cite{L^p BM,log,KolesnikovMilman,Saroglou,Ramon,Milman,Dongmeng}, the sharp affine $L^p$ Sobolev inequality, see \cite{MR1987375}, as well as other sharp affine isoperimetric inequalities, see \cite{HaberlSchuster, LYZ00jdg, MR1987375}. We refer the reader to Chapters 8 and 9 of Schneider's textbook \cite{S14} and to influential articles \cite{BHP17jdg,BLYZ13jams,CW06adv,Lp Aleks,HLYZ16,HZ18adv,LutwakLp,LO95jdg,LYZ00jdg,LYZ04tams,LYZ06imrn,LYZ16,Ol2,Ol21,Oliker,Sta1,YZCVPDE,YZJDG,Zhao,Zu,Zu2,GaussProb,Milman} for an overview of Minkowski problems.  Readers interested in the regularity of solutions, should consult the works of Cheng--Yau \cite{Cheng}, Caffarelli \cite{Caffarelli}, Nirenberg \cite{Nirenberg}, Pogorelov \cite{Pogorelov}, and Trudinger--Wang \cite{Trudinger}. 

The Aleksandrov problem, in particular, has been investigated by numerous people. Oliker \cite{Oliker} and Bertrand \cite{Bertrand}  provided alternative proofs of Aleksandrov's results \cite{Aleks,Aleks0,Aleks1}. The $L_p$ analogues of the Aleksandrov problem were considered by Huang, Lutwak, Yang and Zhang in \cite{Lp Aleks}; by Mui in \cite{Stephanie}; by Zhao in \cite{Zhao}; and by C. Wu, D. Wu, and Xiang in \cite{Lp Gauss}. As for the Gauss Image Measure, the pullback of various $\lambda$ by the radial Gauss Image map, the corresponding measure characterization problem was addressed in \cite{GIP} by B\"or\"oczky, Lutwak, Yang, Zhang and Zhao under the additional assumption of absolute continuity of  $\lambda$ with respect to the spherical Lebesgue measure. It was also addressed by Bertrand in \cite{Bertrand,Bertrand2} using optimal mass transport methods. Finally, Semenov in \cite{Semenov} studied the problem using a mixture of smooth and discrete assumptions on measures.  	
% Considering the dual problem, this simply stated question can be seen to be very similar in spirit to the famous question posed by Minkowski for polytopes: the Minkowski problem asks whether there exist a polytope with specified directions and measures of the facets. We ask about the existence of a polytope with specified directions of the facets, such that each facet is penetrated by exactly one out of specified vectors in its interior. 
\subsection{Necessary and Sufficient Conditions}\label{necc and suff disuss}

Suppose we are given discrete spherical measures $\mu$ and $\lambda$:
\begin{equation}\label{Sept 27.111}\begin{split}
  \lambda=\sum\limits_{j=1}^k\lambda_j \delta_{u_j} \\
  \mu=\sum\limits_{i=1}^m \mu_i\delta_{v_i},
  \end{split}
\end{equation} where $\delta_{v_i}$ and $\delta_{u_j}$ denote Dirac measures of single vectors $v_i\in \sn $ and $u_j\in \mathbb{S}^n$, respectively. The problem, in its most general form, has a natural algebraic obstacle. First, let us note that in the Discrete Gauss Image Problem we can restrict our attention to polytopes with verticies in the radial directions of $v_i$, see Proposition \ref{exist polytope}. Note, as well, that if a polytope $P$ with vertices $r_P(v_i)$  is a solution to the Discrete Gauss Image Problem, then the $u_j$'s are contained in the interiors of the normal cones at the vertices. Otherwise, if some $u_j$ were contained on the boundaries of the normal cones, $\lambda(P,\cdot)$ would fail to be a measure, see Example \ref{example of measure not additive}.
Define a function $f$ from $\{1,\ldots, k\}\rightarrow \{1,\ldots, m\}$ such that $f(j)=i$ if and only if $u_j\in \interior{\balpha_P(v_i)}$. Then, if $P$ is a solution, that is, $\mu(\cdot)=\lambda(P,\cdot):=\lambda(\balpha_P(\cdot))$, we obtain for each $i$: 
\begin{equation}\label{Aug.08.19.347}
 \sum\limits_{j\in f^{-1}(i)} \lambda_j = \mu_i .
\end{equation}

Given two discrete measures, a function $f$ satisfying \eqref{Aug.08.19.347}, assigning "normals" $u_j$ to "radial directions of vertices" $v_i$, might not exist in principle due to purely number-theoretic reasons. (For example, if $k=m=3, \lambda_1=1, \lambda_2=2, \lambda_3=3$, and $\mu_1=\mu_2=\mu_3=2$) Therefore, the Discrete Gauss Image Problem effectively has two steps: 
\begin{itemize}
	\item Prerequisite Partition Problem: What are the necessary and sufficient conditions on the measures so that there exists an assignment function $f:\{1\ldots k\}\rightarrow \{1\ldots m\}$ assigning vectors $u_j$ to vectors $v_i$ satisfying \eqref{Aug.08.19.347}?	
\item Geometric Problem: Given an assignment function $f$  satisfying \eqref{Aug.08.19.347}, when does there exist a polytope  $P$  such that $u_j\in \interior{\balpha_P(v_{f(j)})}$, so that $\mu(\cdot) = \lambda(\balpha_P(\cdot)) $? If not, can we guarantee that for at least one of the assignment functions  $f$  satisfying \eqref{Aug.08.19.347}, such a polytope  $P$  exists?
\end{itemize}

\smallskip
In this work, our goal is to concentrate on the geometric problem. A sufficient natural condition that preserves all varieties of the geometric problem is to assume that $\mu$ is a measure with positive integer weights and $\lambda$ is an equal-weight discrete measure with all weights equal to one. We highlight that these assumptions are general enough for geometric purposes: at the very end of Section \ref{Section Assignment}, see Comment \ref{Remark for general problem}, we show that starting with any collection of vectors $\{u_j\}$ and $\{v_j\}$ and an assignment function $f$ between them, one can associate an equal-weight measure $\lambda$ and a discrete measure $\mu$ with positive integer weights such that \eqref{Aug.08.19.347} is satisfied for this particular $f$. Moreover, given any non-equal-weight problem, one can essentially reduce it to a series of equal-weight problems. See Comment \ref{Remark for general problem}.

 The first major result of this paper is the introduction of what we refer to as the \textit{Assignment Problem}, and the establishment of its equivalence to the Discrete Gauss Image Problem. One can think of the Assignment Problem as the discrete version of the variational equation associated with the Aleksandrov Problem. For a more formal statement, see Section \ref{Section Assignment}, and, in particular, Theorem \ref{main},  where the aforementioned equivalence is established.
 
\smallskip
\noindent{\textbf{The Assignment Problem}} \textit{Suppose $\lambda$ is a discrete equal-weight measure on $\sn$, and $\mu$ is a discrete measure on $\sn$. For each assignment function $f:\{1\ldots k\}\rightarrow \{1\ldots m\}$, satisfying \eqref{Aug.08.19.347}, we define} 
\begin{equation}
   A(f):=\sum_{j=1}^k\log{(u_j \cdot v_{f(j)})}, 
\end{equation}
\textit{where we assume that $\log(x) = -\infty$ for $x \leq 0$.}

\textit{What are the necessary and sufficient conditions on $\mu$ with respect to $\lambda$ such that 
\begin{itemize}
	\item $A(\cdot)$ is maximized by exactly one function $f$, satisfying \eqref{Aug.08.19.347},
	\item for this $f,$ $A(f)>-\infty$?
\end{itemize}}
\smallskip

To address the Assignment Problem as well as the Gauss Image Problem, we introduce two new critical concepts: the \textit{weak Aleksandrov relation}, a necessary condition for both problems that addresses the "$-\infty$" part of the Assignment Problem; and the \textit{edge-normal loop} of two given measures, which connects the uniqueness of the maximizer of $A(\cdot)$ to the geometry of bodies.  

\textit{On the Weak Aleksandrov relation.} When characterizing integral curvature, Aleksandrov discovered what we refer to as the Aleksandrov relation, a necessary and sufficient geometric condition for a measure to be the integral curvature of some body, see \cite{Aleks,Aleks0,Aleks1,GIP} and Definition \ref{strong Aleksandrov related}. While this condition is necessary and sufficient for smooth measures, it turns out to be far from necessary when the problem is considered in a more general setting, see Example \ref{weak Aleksandrov example}. To this end, we introduce the weak version of the Aleksandrov condition. Moreover, we establish that this condition is, in fact, a necessary condition for the Discrete Gauss Image Problem, see Section \ref{Weak Aleksandrov Condition} for this and related results. Additionally, by reinterpreting the assignment function as a matching in a bipartite graph with vertices $\{u_j\}$ and $\{v_i\}$, we show that this condition is equivalent to the assumptions of Hall's Marriage Theorem, see Lemma \ref{Hall}. We discuss the geometric interpretation of the Weak Aleksandrov relation and related results in Section \ref{Weak Aleksandrov Condition}.

\smallskip
\smallskip
\noindent{\textbf{Weak Aleksandrov relation}} 
	Two discrete measures $\mu$ and $\lambda$ on $\sn$ are said to be weak Aleksandrov related if 
 	\begin{equation}
  \lambda(\sn)=\mu(\sn)\geq\mu(\omega)+\lambda(\omega^*)
\end{equation}
for each compact, spherically convex set $\omega\subset\sn$, where $\omega^*$ is the spherical polar set, see \eqref{polarset}.
\smallskip
\smallskip

%TODO STOPED CHECKINg HERE

\textit{On the Edge-normal loop.} The intuition behind the requirement for the uniqueness of the maximizer of $A(\cdot)$ in the Assignment Problem is discussed at the beginning of Section \ref{The Assignment Problem from Geometric Point of View}. We will only briefly mention that this is a an essential property of the Discrete Gauss Image Problem, which does not allow for "splitting of the atom," in contrast to the Aleksandrov Problem. (For example, note that $\lambda(\balpha_P(\cdot))$ ceases to be a finitely additive function, and hence fails to be a measure, if a vector-atom of $\lambda$ is normal to an at least two-dimensional facet of $P$. This situation corresponds to splitting the vector-atom between at least two normal cones of vertices of $P$. See Example \ref{example of measure not additive} and the beginning of Section \ref{The Assignment Problem from Geometric Point of View}.) The nature of this geometric condition comes from tracing all possible loops of edges on a potential polytope solution for the Gauss Image Problem. For a more formal statement and the intuition behind this condition, see Section \ref{The Assignment Problem from Geometric Point of View}, and Definition \ref{Sept 17 Defenition}.

\smallskip
\smallskip
\noindent{\textbf{Edge-normal loop free measures}} Discrete measures $\mu$ and $\lambda$ are called edge-normal loop free if there does not exist a piecewise linear loop in $\mathbb R^n$ such that every vertex lies on a different ray $\{t v_i\mid t>0\}$ and each line segment is perpendicular to a different $u_j$. 

\smallskip
\smallskip

\subsection{Main Result}

We are now ready to state the main result of the paper. In the following, to the spherical discrete measure
\begin{equation}
	\mu=\sum\limits_{i=1}^m \mu_i\delta_{v_i},
\end{equation} where $\delta_{v_i}$ is the Dirac measure at the point $v_i\in\sn$, we associate the set of polytopes $\po_\mu\subset\kno$ (as long as $\mu$ is not concentrated on a closed hemisphere) with vertices in the radial directions $v_i$:
\begin{equation}
	\po_\mu =\{\text{conv}\{\beta_iv_i\mid 1 \leq i \leq m\}\mid (\beta_1,\ldots,\beta_m)\in \mathbb R^m_{>0}\}.
\end{equation}
For other definitions, see Section \ref{Preliminaries}.  

\begin{thm}[Existence]\label{intro main statement}
	Suppose $\mu$ and $\lambda$ are Borel measures on $\sn$, such that $\mu$ is discrete and not concentrated on a closed hemisphere, and $\lambda$ is a discrete equal-weight measure.  Suppose $\mu$ and $\lambda$ are edge-normal loop free. Then, there exists a convex body $K\in \kno$ solving the Discrete Gauss Image Problem, that is, satisfying $\mu(\cdot)=\lambda(\balpha_K(\cdot)) $, if and only if $\mu$ and $\lambda$ are weak Aleksandrov related. Furthermore, we can assume  $K$  is a convex polytope belonging to  $\po_\mu$. 
\end{thm}

For the necessity of the weak Aleksandrov relation see Proposition \ref{necessary weak}. For the other direction, see Corollary \ref{Proof of main corollary}, which combines Theorem \ref{main} (establishing equivalence between the Discrete Gauss Image Problem and the Assignment Problem) and Theorem \ref{main edge normal} (solving the Assignment Problem). For the convex polytope part see Proposition \ref{exist polytope}.

The equal-weight assumption, the weak Aleksandrov relation, and the edge-normal loop free condition were introduced and discussed in Subsection \ref{necc and suff disuss}. Here, we want to briefly comment on how far the edge-normal loop free condition is from being a necessary assumption. In fact, two measures are almost always edge-normal loop free (outside of a low dimensional algebraic variety, see Section \ref{Generic Statements} for details). The rough idea is that two measures are edge-normal loop free if and only if all assignment functions $f$ have different values $A(f)$, see Section \ref{The Assignment Problem from Geometric Point of View}. To solve the Gauss Image Problem, we only need to know from the Assignment Problem (see Theorem \ref{main}) that the assignment function maximizing $A(\cdot)$ is unique, yet we notice that, a priori, we do not have any information on which assignment function will maximize $A(\cdot)$. We show that the edge-normal loop free condition is generic in both the measure-theoretic and algebraic senses (as the complement of a manifold), as it is related to a finite system of equations, see Theorem \ref{Genetric Theorem} and equation \eqref{Sept 17.111}. 

%TDOD Check from here. 

 Since the solution to the Discrete Gauss Image Problem is always non-unique when it exists, see Proposition \ref{pertub}, the proper way to address the uniqueness question is to consider the uniqueness of the assignment function. To this end, we show that the assignment function is unique for a fixed pair of measures. 
 
 In the following, measures $\mu$ and $\lambda$ are given by \eqref{Sept 27.111}. We also implicitly assume that $\lambda(\balpha_K(\cdot))$ and $\lambda(\balpha_L(\cdot))$ are measures. For a formal statement regarding the uniqueness, see Proposition \ref{uniq discrete} and Theorem \ref{main}, as well as Section \ref{Preliminaries} and Subsection \ref{necc and suff disuss} which discuss the conditions under which $\lambda(\balpha_K(\cdot))$ and $\lambda(\balpha_L(\cdot))$ are measures.

 \begin{thm}[Uniqueness]
 	Suppose $\lambda$ and $\mu$ are discrete measures as in \eqref{Sept 27.111}, and let $K,L\in\kno$. Then $K$ and $L$ solve the Discrete Gauss Image Problem; that is, 
 	\begin{equation}
 		\mu=\lambda(\balpha_K(\cdot))=\lambda(\balpha_L(\cdot)),
 	\end{equation}
 	if and only if $K$ and $L$ have the same assignment function; that is, for each $i\in \{1,\ldots, m\}$ and $j\in \{1,\ldots, k\}$ the following holds: 
 	\begin{equation}
 		u_j\in \interior{\balpha_K(v_{i})} \Leftrightarrow u_j\in \interior{\balpha_L(v_{i})}.
 	\end{equation}
 \end{thm}

\subsection{Remarks about proofs.}

In problems of this type, one can usually construct a functional on the set of convex bodies such that the body maximizing the functional is a solution to the problem. For these types of problems, it is usually relatively easier to verify that the convex body maximizing the functional is a solution to the given problem, yet it is often significantly harder to show the existence of the body maximizing the functional. Such a functional, see Equation \eqref{GIP functional}, is indeed applicable to this problem. However, the opposite occurs: it is relatively easy to see that there exists a body maximizing the functional, but the maximizing body is not necessarily a solution. This is exactly the opposite of what happens when assuming that $\lambda$ is absolutely continuous, as in \cite{GIP}, for example. The difficulty lies in the interior part of the condition $u_j\in \interior{\balpha_K(v_{j})}$.

Our first goal is establishing the equivalence between the Discrete Gauss Image Problem and the Assignment Problem, see Theorem \ref{main}.  We give two different proofs of this equivalence. The main proof, see Section \ref{Main section}, has the following structure: first, we construct a sequence of absolutely continuous measures $\lambda_\eps\rightarrow\lambda$, which are weak Aleksandrov related to $\mu$. Then we show the existence of a solution $K_\eps$ for the $\lambda_\eps,\mu$-problem under the weak Aleksandrov assumption. We then study the subsequence convergence of $K_\eps$ and find conditions under which the limiting body is the solution to the original question. The last part relies on a generalization of the Birkhoff-von Neumann theorem and transportation polytopes. 

We also give an alternative proof of the equivalence between the two problems, for the special case when both measures are equal-weight, in Section \ref{Alternative Approach}. It is shorter but more computational. It uses some calculations from the recent work of Wyczesany on optimal mass transport \cite{Wyczesany}. The connection here is not unexpected, as one can reformulate the Gauss Image Problem as a mass transport problem, see Bertrand \cite{Bertrand} and Oliker \cite{Oliker}. Moreover, the results in \cite{Wyczesany} advance, among other things, the well-known cyclic monotonicity condition introduced by Rockafellar \cite{cycl}.
 
In Sections \ref{The Assignment Problem from Geometric Point of View} and \ref{Generic Statements}, we analyze the Assignment Problem to derive conclusions for the Discrete Gauss Image Problem, discussing the Edge-normal loop condition. In Section \ref{Weak Aleksandrov Condition}, we discuss the Weak Aleksandrov condition. Section \ref{Section Assignment} is devoted to setting up the equivalence between the Gauss Image Problem and the Assignment Problem.

\section{Preliminaries}\label{Preliminaries}
Let $\kn$ be the set of convex bodies in $\mathbb R^n$, that is compact convex sets with nonempty interior. By $\kno\subset \kn$, we denote those convex bodies which contain origin in their interior. By $\partial K$, we denote the boundary of $K$. Given $K\in\kno$, \textit{the radial map}  $r_K:\sn\rightarrow\partial K$ is defined by 
\begin{equation}
r_K(u):=ru\in\partial K,
\end{equation} 
for some positive $r$. By $N(K,x)$, we denote the \textit{normal cone of $K$ at} $x\in\partial K$, that is the set of all outer unit normals at $x$:
\begin{equation}
  N(K,x)=\{v\in\sn : (y-x)\cdot v\leq 0 \text{ for all } y\in K \}.
\end{equation}

Given $K\in\kno$, we define the radial Gauss image of $\omega\subset\sn$ as: 
\begin{equation}\label{radial gauss map defenition}
\balpha_K(\omega)=\bigcup_{x\in r_K(\omega)}N(K,x)\subset \sn	.
\end{equation}
The radial Gauss Image map, $\balpha_K$, maps sets of $\sn$ to sets of $\sn$. Outside of a spherical set of Lebesgue measure zero, the multivalued map $\balpha_K$ is singular valued. It is known that $\balpha_K$ maps Borel measurable sets to Lebesgue measurable sets. See p.88--89 in \cite{S14} for both of these results.  We denote the restriction of $\balpha_K$ to the corresponding singular valued map by $\alpha_K$. For additional details, we refer the reader to \cite{GIP}. 

When working with the radial Gauss Image, we will often abuse the notation, and write $\balpha_K(v)$  directly, instead of more proper $\balpha_K(\{v\})$, if $v$ is a vector. Moreover, by 
\begin{equation}
	\interior{\balpha_K(\omega)}
\end{equation}   
we denote the interior of the set $\balpha_K(\omega)\subset \sn$.

Suppose $\lambda$ is a Borel measure on $\sn$  and $K\in\kno$. Then $\lambda(K,\cdot)$,  \textit{the Gauss image measure of $\lambda$ via $K$}, is a submeasure defined as the pullback of $\lambda$ via the map $\balpha_K$. That is, for each Borel $\omega\subset\sn$, 
\begin{equation}\label{Aug.19.121}
 \lambda(K,\omega):= \lambda(\balpha_K(\omega)) .
\end{equation}
If $\lambda$ is absolutely continuous, $\lambda(K,\cdot)$ is a Borel measure \cite{GIP}. Note, however, that if $\lambda$ is not absolutely continuous, which is the case of this work, $\lambda(K,\cdot)$ might not even be finitely additive. 
\smallskip

\begin{ex}\label{example of measure not additive}
	let $K$ be a square centered at the origin with sides perpendicular to the unit vectors $u_1,u_2,u_3,u_4$. Let $\lambda=\sum\limits_{i=1}^4\delta_{u_i}$ where $\delta_{u_i}$ are Dirac measures of sets $\{u_i\}$. Let the unit vectors $v_1$ and $v_2$ be such that boundary points $r_K(v_1),r_K(v_2)$ are in the interior of the side of $K$ perpendicular to $u_1$. Then, \begin{equation}\label{Sept 7.1}
	\balpha_K(\{v_1\})=\balpha_K(\{v_2\})=\balpha_K(\{v_1,v_2\})=\{u_1\}.
\end{equation}
Implying that:
\begin{equation}\label{Sept 7.2}
	1=\lambda(K,\{v_1\})=\lambda(K,\{v_2\})=\lambda(K,\{v_1,v_2\}),
\end{equation}
which establishes that $\lambda(K,\cdot)$ is not finitely additive.
\end{ex} 
Yet, we notice that for discrete measure $\lambda=\sum\limits_{j=1}^k\lambda_j \delta_{u_j}$ and $K\in\kno$,
the Gauss Image Measure $\lambda(K,\cdot)$ is a measure if and only if for each $j,$ $u_j$ is normal to a single point on the boundary of $K$. We also note that
\begin{equation}
	\lambda(K,\cdot)=\lambda(cK,\cdot) 
\end{equation}
for any scalar $c>0$, as $\balpha_K=\balpha_{cK}.$

We also note Gauss image measure is a valuation, which has been is proven in \cite{GIP}. We refer the reader to the articles \cite{L1,L2,L3,L4,Sh1,Sh2,Sh3,Al1,Al2,Al3} on the theory of valuations. 

%TODO add

\textit{The radial function} $\rho_K:\sn\rightarrow \mathbb R$ is defined by: \begin{equation}
  \rho_K(u)=\max\{a : au\in K\}.
\end{equation}
In this case, $r_K(u)=\rho_K(u)u$. \textit{The support function} is defined by \begin{equation}
  h_K(x)=\max\{x\cdot y : y\in K\}.
\end{equation}
 For $K\in\kno$ we define its \textit{polar body} $K^*$ as the unique convex body associated with the support function $h_{K^*}:=\frac 1 {\rho_K}$. 
 
  We denote by $\mathfrak r_K$ the radius of the largest ball contained in $K$ and centered at $o$. Similarly, we denote $\mathfrak R_K$ to be the radius of the smallest ball containing $K$ and centered at $o$. We will refer to $\mathfrak r_K$ as \textit{the inner radius} of the body $K$, to $\mathfrak R_K$ as \textit{the outer radius} of $K$, and to the ratio $\frac {\mathfrak r_K} {\mathfrak R_k}$ as \textit{the inner to outer radius ratio} of the body $K$.
 
 It is important to note that for any $K\in\kno$, the following identity holds: \begin{equation}\label{some geometric obvious relation}
 	\min{\rho_K}=\min{h_K}=\mathfrak r_K\leq \mathfrak R_k=\max{\rho_K}=\max{h_K}. 
 \end{equation}\textit{The support hyperplane} to $K$ with an outer unit normal $v\in\sn$ is defined as \begin{equation}
 H_K(v)=\{x:x\cdot v=h_K(v)\}.	
 \end{equation}
By $H^-(\alpha,v)$ we denote the halfspace $\{x:x\cdot v\leq\alpha\}$ and by $H(\alpha,v)$ we denote the hyperplane $\{x:x\cdot v=\alpha\}$. Given a set $S\subset  \rn$ we write its convex hull as
\begin{equation}
	\text{conv}(S).
\end{equation}

For $\omega \subset S^{n-1}$, we define $\text{cone}\,\omega \subset \rn$,
the {\it cone that $\omega $ generates}, as
\begin{equation}
  \text{cone}\,\omega = \{tu:\text{$t\ge 0$ and $u\in\omega$}\}.
\end{equation}
And define $\hat \omega$, the {\it restricted cone that $\omega $ generates}, as
\begin{equation}
  \hat \omega =
\{tu : \text{$0\le t \le 1$ and $u\in \omega$}  \}.
\end{equation}
We are going to say that $\omega \subset S^{n-1}$ is {\it spherically convex} if the
cone that $\omega $ generates is a nonempty convex subset of $\rn$,
that is not all of $\rn$.
Therefore,
 a spherically convex set on $\sn$ is nonempty and
 is always contained in a closed hemisphere of $\sn$. For a subset $\omega\subset S^{n-1}$ which is contained in a closed hemisphere, we define
the {\it spherical convex hull}, $\bla \omega \bra$, of $\omega$, by
\begin{equation}
\bla \omega \bra =\sn\cap\text{conv}\,(\text{cone}\,\omega).
\end{equation}
 Given $\omega\subset S^{n-1}$ contained in a closed hemisphere,
 {\it polar set} $\omega^*$ is defined by:
\begin{equation}\label{polarset}
\begin{split}	
\omega^* &=
\bigcap_{u\in\omega}\{v\in S^{n-1} : u\cdot v\leq 0 \}. 
\end{split}
\end{equation}
We note that polar set is always convex. If $\omega\subset\sn$ is a closed set, and $\alpha$ is a constant, we define its {\it outer parallel set} $\omega_\alpha$, as
\begin{equation}
	\omega_\alpha =
 \bigcup_{u\in \omega} \{v\in \sn : u\cdot v > \cos\alpha \}.
\end{equation}
For recent work on spherical convex bodies, see Besau and Werner \cite{Besau}.

 Given $K\in\kno$ and measures $\mu$ and $\lambda$, we define the functional $\Phi(K,\mu,\lambda)$ by 
\begin{equation}\label{GIP functional}
	\Phi(K,\mu,\lambda):=\int\log\rho_Kd\mu+\int\log\rho_{K^*}d\lambda.
\end{equation}
Note that $\Phi(K,\mu,\lambda)=\Phi_{\mu,\lambda}(K^*)$ in \cite{GIP} notation. It is important to stress that:
\begin{equation}\label{scaling motivation}
\begin{split}
	&\text{If } \mu=\lambda(K,\cdot) \text{, then } \mu=\lambda(cK,\cdot) \text{ for any }c>0.\\
	&\Phi(K,\mu,\lambda)=\Phi(cK,\mu,\lambda) \text{ for any } c>0.
\end{split}
\end{equation}
That is, the nature of the problem is not sensitive to the rescaling of the convex bodies.

	 We call measures $\lambda$ and $\mu$ on $\sn$ \textit{discrete} if they have the form: 
\begin{equation}\begin{split}
  \lambda=\sum\limits_{j=1}^k\lambda_j \delta_{u_j} \\
  \mu=\sum\limits_{i=1}^m \mu_i\delta_{v_i},
  \end{split}
\end{equation}
where $\delta_{u_j}$ and  $\delta_{v_i}$ are the Dirac measures of sets $\{u_j\}$ and $\{v_i\}$ containing single vectors $u_j,v_i\in \sn$, and coefficients $\lambda_j,\mu_i\in\mathbb N$. Discrete measure $\lambda$ is called \textit{equal-weight} if $\lambda_j=1$ for all $j$. Similarly, for $\mu$. Note that we only deal with finitely many weights. The notation is consistent: \begin{equation}\begin{split}
	j,u_j,\lambda_j,k \text{ are always associated with }\lambda, \\
	 i,v_i,\mu_i,m \text{ are always associated with } \mu.
\end{split}
\end{equation} Given a discrete measure $\mu$ not concentrated on a closed hemisphere, we denote by $\po_\mu\subset\kno$ the following set of polytopes:
\begin{equation}\label{convex hull notation}
	\po_\mu =\{\text{conv}\{\beta_iv_i\mid 1 \leq i \leq m\}\mid (\beta_1,\ldots,\beta_m)\in \mathbb R^m_{>0}\}.
\end{equation}
 That is, $\po_\mu$ is the set of convex polytopes containing the origin in their interiors where each vertex can be written as $r_P(v_i)$ for some $i$. Note that if $P\in \po_\mu,$ it can as well be written as the following:
\begin{equation}
  P=(\bigcap_{i=1}^mH^-(\alpha_i,v_i))^* ,
\end{equation}
where $\alpha_i>0$. 

If for a given $\mu$ and $\lambda$, there exists $K\in\kno$ such that $\mu=\lambda(K,\cdot)$, we say that measures $\mu$ and $\lambda$ are related by the convex body $K$. By $\bar\lambda$, we denote the Lebesgue measure on $\sn$. Two Borel measures $\mu$ and $\lambda$ on $\sn$ are called \textit{Aleksandrov related} if 
 	\begin{equation}
  \lambda(\sn)=\mu(\sn)>\mu(\omega)+\lambda(\omega^*),
\end{equation}
for each compact, spherically convex set $\omega\subset\sn$. Two discrete measures $\mu$ and $\lambda$ on $\sn$ are called \textit{weak Aleksandrov related} if 
 	\begin{equation}
  \lambda(\sn)=\mu(\sn)\geq\mu(\omega)+\lambda(\omega^*)
\end{equation}
for each compact, spherically convex set $\omega\subset\sn$. We note the interchangeable use of the terms "Aleksandrov condition" and "Aleksandrov relation". We discuss these conditions in more detail in Section \ref{Weak Aleksandrov Condition}.

We also would like to mention several results related to the Gauss Image Problem. When $\lambda$ is absolutely continious, the following was obtained by K. J. B\"or\"oczky, E. Lutwak, D. Yang, G. Y. Zhang and Y. M. Zhao \cite{GIP} and, separately and in a different way, by Bertrand \cite{Bertrand}:
\begin{thm}
Suppose $\mu$ and $\lambda$ are Borel measures on $\sn$ and $\lambda$ is absolutely continuous. If $\mu$ and $\lambda$ are Aleksandrov related, then there exists a $K\in\kno$ such that $\mu=\lambda(K,\cdot)$.
\end{thm} 
Moreover, it was shown that the Aleksandrov relation is a necessary assumption for the existence of a solution to the Gauss Image problem, if one of the measures is assumed to be absolutely continuous and strictly positive on open sets \cite{GIP}. In this case, the solution to the Gauss Image problem is shown to be unique up to a dilation. 

The results with mixed discrete and smooth assumptions for measures were established by Semenov in \cite{Semenov}, and are used in the proof of Theorem \ref{main}. Since the mixed problem uses slightly different version of the weak Aleksandrov relation (which is appropriate for measures without any additional discrete or smooth assumptions), we mention these results in the appendix, see Theorem \ref{appendix main} and Lemma \ref{weak Aleks bound}, to not cause confusion with weak Aleksandrov relation for discrete measures.

We use the books of Schneider \cite{S14} as our standard reference. The books of Gruber and Gardner are also good alternatives \cite{G06book,Gruberbook}.

\section{Weak Aleksandrov Condition}\label{Weak Aleksandrov Condition}

Let $\mu$ and $\lambda$ be discrete equal-weight measures. Suppose there exists a solution $P\in\kno$ such that $\mu=\lambda(P,\cdot)$. Clearly, this implies $m=\mu(\sn)=\lambda(\sn)=k$. Note that if $u\in \{u_1,\ldots ,u_m\}$ is contained in normal cone of a vertex at direction $v \in \{v_1,\ldots ,v_m\}$, then $uv>0$ since $P\in\kno$. (More formally, since $u$ is normal to the boundary point $r_P(v)$ of $P,$ and $P$ contains the origin at the interior, we obtain that $r_P(v)u=h_P(u)>0$. Since $r_P(v)=\rho_P(v)v,$ we obtain $\rho_P(v)uv>0,$ from which it follows that $uv>0$.)

So, before we even attempt to find the solution for the given measures, we should guarantee the existence of a pairing between two sets of vectors such that $uv>0$ in each pair. This leads us two to questions: 
\begin{itemize}
	\item Does there exist a good necessary assumption that guarantees the existence of a pairing between vectors $\{u_1\ldots u_m\}$ and $\{v_1\ldots v_m\}$ such that $uv>0$ for each pair?
	\item For a specified pairing, does there exist a solution?
\end{itemize}
  The answer to the first part turns out to be the weak Aleksandrov condition.

 \begin{defi}\label{strong Aleksandrov related}
 	Two Borel measures $\mu$ and $\lambda$ on $\sn$ are called \textit{Aleksandrov related} if 
 	\begin{equation}
  \lambda(\sn)=\mu(\sn)>\mu(\omega)+\lambda(\omega^*)
\end{equation}
for each compact, spherically convex (See Section \ref{Preliminaries}) set $\omega\subset\sn$. In above, the set $\omega^*$ is defined as a polar set: 
\begin{equation}
\omega^*:=\bigcap_{u\in\omega}\{v\in\sn : uv\leq 0\}	
\end{equation}
\end{defi}

\begin{defi}\label{weak Aleksandrov related}
	Two discrete measures $\mu$ and $\lambda$ on $\sn$ are called \textit{weak Aleksandrov related} if 
 	\begin{equation}
  \lambda(\sn)=\mu(\sn)\geq\mu(\omega)+\lambda(\omega^*)
\end{equation}
for each compact, spherically convex set $\omega\subset\sn$.
\end{defi}

Note that we use the phrases  \textit{"Aleksandrov condition"} and \textit{"Aleksandrov relation"} interchangeably. Since for any $\omega\subset \sn$ a compact spherically convex set, $\sn\setminus\omega^*=\omega_{\frac\pi2}$ and $\omega^{**}=\omega$, we immediately obtain the following equivalent definitions:

\begin{prop}
Two Borel measures $\mu$ and $\lambda$ on $\sn$ are Aleksandrov related if and only if $\mu(\sn)=\lambda(\sn)$, and for each compact spherically convex set $\omega\subset\sn$, \begin{equation}
 	\mu(\omega)<\lambda(\omega_{\frac\pi2}).
 \end{equation}
\end{prop}

\begin{prop}\label{Aleks Relation}
 Two discrete Borel measures $\mu$ and $\lambda$ on $\sn$ are weak Aleksandrov related if and only if $\mu(\sn)=\lambda(\sn)$, and for each compact spherically convex set $\omega\subset\sn$, \begin{equation}
 	\mu(\omega)\leq\lambda(\omega_{\frac\pi2}).
 \end{equation}
 Or, alternatively, for each compact spherically convex set $\omega\subset\sn$,
 \begin{equation}
 	\lambda(\omega)\leq\mu(\omega_{\frac\pi2}).
 \end{equation}
\end{prop}
\begin{rmk}
Sometimes we will write Strong Aleksandrov related in place of Aleksandrov related to emphasize the difference.	
\end{rmk}

The difference between the two conditions is demonstrated by the following example: 

\begin{ex}\label{weak Aleksandrov example}
Take any equilateral triangle $P$ in $\mathbb{R}^2$ centered at the origin with vertices on the unit sphere. Let $v_1,v_2,v_3$ be such that $r_P(v_i)=v_i$ are different vertices of the triangle. Let $\mu=\lambda$ be the discrete equal-weights measure:\begin{equation}
  \lambda=\mu=\sum\limits_{i=1}^3\delta_{v_i}.
\end{equation}
Then, clearly, $\mu=\lambda(P,\cdot)$. Note that $\mu$ and $\lambda$ are weak Aleksandrov related but not strong Aleksandrov related. It is also interesting to note that for these particular measures, any triangle   $P\in\po_\mu$  is a solution. We leave the details to the reader.
\end{ex}

The next proposition provides a necessary condition for two measures to be related by a convex body. In particular, it shows that the weak Aleksandrov relation is a necessary condition for two discrete measures to be related by a convex body. 

\begin{prop}\label{necessary weak}
Given two Borel measures $\mu$ and $\lambda$ on $\sn$, suppose they are related by a convex body $K\in\kno$. That is, $\mu=\lambda(K,\cdot)$. Then $\mu(\sn)=\lambda(\sn)$, and there exists a uniform $\alpha>0$ such that for each compact, spherically convex set $\omega\subset\sn$, 
	\begin{equation}\label{Sept 22 more refined statement}
		\mu(\omega)\leq\lambda(\omega_{\frac\pi2-\alpha}).
	\end{equation}	
	Moreover, if $\mu$ and $\lambda$ are discrete, then they are weak Aleksandrov related.
\end{prop}
\begin{proof}
	Since $K\in\kno$, there exists a $c>0$ such that $\frac{\mathfrak r_K}{\mathfrak R_K}>c$. Consider some $u\in\sn$ and $v\in\balpha_K(u)$. Then from \eqref{some geometric obvious relation},
	\begin{equation}
	\mathfrak r_K\leq h_K(v)= \rho_K(u)uv \leq \mathfrak R_Kuv. 	
	\end{equation}
Hence, $c<\frac{\mathfrak r_K}{\mathfrak R_K}\leq uv$. So, for each $u\in\sn$, we have \begin{equation}
  \balpha_K(u)\subset u_{ \arccos(c)}\subset u_{\frac \pi 2 - \alpha}, 
  \end{equation}
  for some $\alpha$ where $0<\alpha<\frac \pi 2$. Therefore, for any compact spherically convex $\omega$, since $\omega_{\frac \pi 2 - \alpha}=\cup_{u\in \omega} u_{\frac \pi 2 - \alpha}$, we obtain  \begin{equation}
  \mu(\omega)=\lambda(K,\omega)=\lambda(\balpha_K(\omega))\leq\lambda(\omega_{\frac \pi 2 - \alpha}). 
\end{equation}
The second part of the claim immediately follows from Proposition \ref{Aleks Relation}.
\end{proof}

We also state this lemma as a trivial consequence of the above proof for the later reference. 

\begin{lem}\label{bound on normal cone}
	Given $K\in\kno$, suppose $0<c<\frac {\mathfrak r_K}{\mathfrak R_k}$. Then for any $v\in\sn$, \begin{equation}
  \balpha_K(v)\subset v_{\arccos c}.
\end{equation}
\end{lem}

Note that for discrete measures, the weak Aleksandrov relation actually implies the above conclusion since every compact spherically convex set is closed:

\begin{prop}\label{uniform constant}
Suppose the discrete measures $\mu$ and $\lambda$ satisfy the weak Aleksandrov relation. Then there exists a uniform $\alpha>0$ such that for each closed set $\omega\subset\sn$: 
\begin{equation}
\mu(\omega)\leq\lambda(\omega_{\frac\pi2-\alpha}).
\end{equation}

\end{prop}
\begin{proof}
First we claim that for each closed set $\omega\subset\sn$ contained in a closed hemisphere, $\omega_{\frac \pi 2}={\bla \omega \bra}_{\frac \pi 2}$. By set inclusion, $\omega_{\frac \pi 2}\subset{\bla \omega \bra}_{\frac \pi 2}$. For the opposite direction, take any $v\in {\bla \omega \bra}_{\frac \pi 2}$. Then for some $x\in\bla \omega \bra$, we have that $xv>0$. By the definition of convex hull, we can write $x$ as some convex combination of finite number of vectors $y_j\in\text{cone}(\omega)$. By the definition of the cone, each $y_j$ is a positive scaling of some vector  $z_j\in \omega$. Hence, $x$ is a linear combination of some vectors $z_j$ with positive coefficients. Therefore, since $xv>0$, we have that for some $z_j$, $z_jv>0$, and thus $v\in \omega_{\frac \pi 2}$. So, $\omega_{\frac \pi 2}\supset{\bla \omega \bra}_{\frac \pi 2}$. And hence, $\omega_{\frac \pi 2}={\bla \omega \bra}_{\frac \pi 2}$.

Let $\mathcal A$ be a collection of all possible indices $I\in\{1\ldots m\}$ such that $\{v_i\}_{i\in I}$ are contained in closed hemisphere. Given  $I\in\mathcal A$, define $\omega^I$ as $\cup_{i\in I}v_i$. Then by Proposition \ref{Aleks Relation} and from the previous conclusion, we obtain: 
\begin{equation}
  \mu(\omega^I)\leq\mu(\bla\omega^I\bra)\leq\lambda({\bla\omega^I\bra}_{\frac \pi 2})=\lambda({\omega^I}_{\frac \pi 2}).
\end{equation}
Since $\lambda$ is a discrete measure and $(v_i)_{\frac \pi 2}$ is an open set, for any $v_i$, where $i\in\{1\ldots m\}$, there exists an $\alpha_i$ such that $\lambda((v_i)_{\frac \pi 2 - \alpha_i})=\lambda((v_i)_{\frac \pi 2})$. Let $\alpha=\min_{i}\alpha_i$. Therefore, by the definition of outer parallel set:

\begin{equation}\label{3.11}
  \lambda({\omega^I}_{\frac \pi 2}\setminus{\omega^I}_{\frac \pi 2 -\alpha})=\lambda(\bigcup_{i\in I}(v_i)_{\frac \pi 2}\setminus \bigcup_{i\in I}(v_i)_{\frac \pi 2-\alpha})\leq\lambda(\bigcup_{i\in I}((v_i)_{\frac \pi 2}\setminus(v_i)_{\frac \pi 2 - \alpha_i}))=0,
\end{equation}
where the middle inequality follows from the set inclusion. Combining the last two equations, we obtain that for any $I\in \mathcal A$,
\begin{equation}
  \mu(\omega^I)\leq \lambda(\omega^I_{\frac \pi 2 -\alpha}).
\end{equation}

Suppose now we are given an index set $I\in\{1\ldots m\}$ such that $\{v_i\}_{i\in I}$ are not contained in a closed hemisphere. Then $\mu(\omega^I)\leq\mu(\sn)=\lambda(\sn)=\lambda(\omega^I_{\frac \pi 2})$. And similarly to \eqref{3.11}, we obtain that $\lambda(\omega^I_{\frac \pi 2})=\lambda(\omega^I_{\frac \pi 2-\alpha})$. Therefore, combining with the previous claim, we obtain that:
\begin{equation}
  \mu(\omega^I)\leq \lambda(\omega^I_{\frac \pi 2 -\alpha})
\end{equation}
for any index set $I\in\{1\ldots m\}$.

Now given any closed set $\omega\subset\sn$, let $\omega^I=\omega\cap\{v_1\ldots v_m\}$. Then from the previous inequality,
\begin{equation}
  \mu(\omega)=\mu(\omega^I)\leq\lambda(\omega^I_{\frac \pi 2 - \alpha})\leq \lambda(\omega_{\frac \pi 2 -\alpha}),
\end{equation}
where the last inequality follows by the set inclusion.
\end{proof}

\section{The Assignment Functional} \label{Section Assignment}

Recall that to the measure $\mu$ not concentrated on a closed hemisphere, we associate the set of polytopes $\po_\mu\subset\kno$ with vertices at the radial directions $v_i$:
\begin{equation}
	\po_\mu =\{\text{conv}\{\beta_iv_i\mid 1 \leq i \leq m\}\mid (\beta_1,\ldots,\beta_m)\in \mathbb R^m_{>0}\}.
\end{equation}
Recall, that each $P\in \po_\mu,$ can also be written in the following way:
\begin{equation}
  P=(\bigcap_{i=1}^mH^-(\alpha_i,v_i))^* ,
\end{equation}
where $\alpha_i>0$.

In the next proposition we show that it suffices to concentrate on the polytopes when considering the Discrete Gauss Image Problem. 

\begin{prop}\label{exist polytope}
Given a discrete measure $\mu$ not concentrated on a closed hemisphere and a discrete measure $\lambda$, suppose there exists a body $K\in\kno$ satisfying $\mu=\lambda(K,\cdot)$. Define $P\in\po_\mu$ to be a polytope with vertices $r_K(v_i)$,
\begin{equation}
\begin{split}
  P&=\text{\normalfont conv}\{r_K(v_i)\mid 1 \leq i \leq m\},\\
  P&=(\bigcap_{i=1}^mH^-(1/r_K(v_i),v_i))^*.
\end{split}
\end{equation}
Then $\mu=\lambda(P,\cdot)$, and for each $j$, $u_j\in \interior{\balpha_K(v_{f(j)})}$, for some function $f\colon\{1\ldots k\}\rightarrow \{1\ldots m\}$.
	\end{prop}
	
	\begin{proof}
%		The existence of a polytope $P$ is justified by simply taking the convex hull of points $r_K(v_i)$. Since $\mu$ is not concentrated on a closed hemisphere, we obtain that $P$ contains the origin in its interior. Hence, $P\in\kno$ and $P\in\po_\mu$.  
		
		First we claim that for any $i_1\neq i_2$, where $i_1,i_2\in\{1,\ldots, k\}$, the set ${\balpha_K(v_{i_1})}\cap{\balpha_K(v_{i_2})}$ does not contain any vectors $\{u_1\ldots u_k\}$. We proceed by contradiction and suppose the claim is not true. Let some $u_j\in {\balpha_K(v_{i_1})}\cap{\balpha_K(v_{i_2})}$. First  note that $v_{i_1}\neq - v_{i_2}$, as for any $v\in\sn$:
	\begin{equation}
   \balpha_K(v)\subset v_{\frac \pi 2},
\end{equation}
from Lemma \ref{bound on normal cone}. And this would mean that $\balpha_K(v)\cap \balpha_K(-v)=0$. Note that the set of boundary points of K for which $u_j$ is normal, \begin{equation}
  H_K(u_j)\cap K
\end{equation}
 is a convex set. Hence, 
 \begin{equation}
  \text{cone}\, (H_K(u_j)\cap K) = \{tx:\text{$t\ge 0$ and $x\in H_K(u_j)\cap K$}\}
\end{equation}
is a convex set. So $r_K^{-1}(H_K(u_j)\cap K)$ is a spherical convex set, as $\text{cone}\, (H_K(u_j)\cap K)\neq \rn$. Since $v_{i_1},v_{i_2}\in r_K^{-1}(H_K(u_j)\cap K)$, we obtain $\bla\{v_{i_1},v_{i_2}\}\bra\subset r_K^{-1}(H_K(u_j)\cap K)$. Therefore, for any $v\in \bla\{v_{i_1},v_{i_2}\}\bra$, we have $u_j\in\balpha_K(v)$. Thus $\lambda(K,v)\geq 1$ for uncountably many $v\in \bla\{v_{i_1},v_{i_2}\}\bra$, and therefore, $\lambda(K,\cdot)$ can't be a finite measure. We arrive at a contradiction.
		
		Therefore, given $j$, we can properly define a function $f$ by $u_j\in {\balpha_K(v_{f(j)})}$. Since for any $i$, $\balpha_K(v_i)\subset\balpha_P(v_i)$, we obtain that $u_j\in {\balpha_P(v_{f(j)})}$. We now check that the interior assumption is satisfied. That is, $u_j\in \interior{\balpha_P(v_{f(j)})}$, which will finish the proof. Suppose it is not. Pick some $u_j\in \partial{\balpha_P(v_{f(j)})}$. Since $P$ is a polytope with vertex $r_P(v_{f(j)})$, there exist some other $v_i$ for $i\neq f(j)$, such that $u_j\in\partial{\balpha_P(v_i)}$. Now, for the support function $h_P$, we obtain \begin{equation}
  h_P(u_j)=u_jr_P(v_i)=u_jr_P(v_{f(j)}).
\end{equation}
Since we showed that $u_j\not\in  {\balpha_K(v_{i_1})}\cap{\balpha_K(v_{i_2})}$ for any $i_1\neq i_2$, we have that $u_j\notin \balpha_K(v_{i})$. Therefore, \begin{equation}
  u_jr_P(v_i)=u_jr_K(v_i)<h_K(u_i),
\end{equation}
where we additionally used that $r_K(v_i)=r_P(v_i)$. However, since $u_j\in\balpha_K(v_{f(j)})$, we have \begin{equation}
  u_jr_P(v_{f(i)})=u_jr_K(v_{f(i)})=h_K(u_i).
\end{equation}
We get the contradiction from the combination of the last three equations.

	\end{proof}

We are now ready to properly introduce the assignment functions we first mentioned in subsection \ref{necc and suff disuss}.

\begin{defi}\label{assign} Given discrete measures $\lambda$ and $\mu$, we define \textit{the set of assignment functions}: 
\begin{equation}
  \mathbb{F_{\mu,\lambda}}:=\Big\{f\colon\{1\ldots k\}\rightarrow \{1\ldots m\} \mid \sum\limits_{j\in f^{-1}(i)} \lambda_j = \mu_i\big \}.
\end{equation}
A function $f\in\mathbb{F_{\mu,\lambda}}$ is callend \textit{an assignment function} with respect to $\mu$ and $\lambda$. If, for $f\in\mathbb{F_{\mu,\lambda}}$, there exists a body $K\in\kno$ solving the Gauss Image problem for measures $\mu$ and $\lambda$, and $u_j\in \interior{\balpha_K(v_{f(j)})}$, then we call $f$ \textit{a solution function}. We also say that $f$, or $f_K$, is an assignment function with respect to the body $K$. Sometimes we write $\mathbb F$ instead of $\mathbb F_{\mu,\lambda}$. Finally, we define the set of the \textit{proper assignment functions} as a subset of the above:
\begin{equation}
  \mathbb{F}_{\mu,\lambda,p}:=\{f\in\mathbb{F} \mid \forall j, u_jv_{f(j)}>0\}.
\end{equation}
Sometimes we write $\mathbb{F}_p$ instead of $\mathbb{F}_{\mu,\lambda,p}$.
\end{defi} 

\begin{rmk}	
Note that if $f\in\mathbb{F}\setminus\mathbb{F}_p$, then $f$ cannot be a solution function. This  follows from the definition and Lemma \ref{bound on normal cone}.
\end{rmk}

Guided by the preceding remark, we claim that if $\mu$ and $\lambda$ are weak Aleksandrov related, then $\mathbb{F}_p\neq \emps$. There are several ways to see this. First, this would be a consequence of the proof of the main Theorem \ref{main}. Alternatively, one could prove the statement by induction. Here, we present a proof for equal-weight $\mu$ and $\lambda$, which provides an interesting connection to graph theory and gives another justification for the weak Aleksandrov relation.

Suppose we are given two discrete equal-weight measures $\mu$ and $\lambda$. Consider a bipartite graph $G(\mu,\lambda)$ formed by two sets of vertices $\{u_j\}$ and $\{v_i\}$ where we draw an edge from $v_i$ to $u_j$ if and only if $v_iu_j>0$. For $V\subset\{v_1\ldots v_m\}$ define $N_G(V)$ as the set of  adjacent verticies to $V$. Hall's marriage theorem states that there exists a perfect matching between verticies of $G(\mu,\lambda)$ if and only if for each subset of vertices $V\subset\{v_1\ldots v_m\}$: 
\begin{equation}
  |V|\leq|N_G(V)|.
\end{equation}
Notice that this is exactly the same as weak Aleksandrov condition:  

\begin{lem}\label{Hall}
Suppose we are given two discrete equal-weight measures $\mu$ and $\lambda$, where $\lambda(\mathbb{S}^{n-1})=\mu(\mathbb{S}^{n-1})$. Then Hall's marriage condition is equivalent to the weak Aleksandrov condition. In particular, if two measures satisfy weak Aleksandrov condition, then $\mathbb{F}_p$ is nonempty. \end{lem}
\begin{proof}
	Suppose for each subset of vertices $V\subset\{v_1\ldots v_m\}$, \begin{equation}
  |V|\leq|N_G(V)|.
\end{equation} Then for any $\omega$ a compact spherically convex set, \begin{equation}\begin{split}
  \mu(\omega)=&|\omega\cap\{v_1\ldots v_m\}|\leq|N_G(\omega\cap\{v_1\ldots v_m\})|\leq \\ &|\omega_{\frac\pi2}\cap\{u_1\ldots u_m\}|=\lambda(\omega_{\frac\pi 2}). 
  \end{split}
\end{equation}
Thus, by Proposition \ref{Aleks Relation}, $\mu$ and $\lambda$ are weak Aleksandrov related. For another direction, suppose the weak Aleksandrov condition holds. Consider any $V\subset\{v_1\ldots v_m\}$. Then $V$ is a closed set, and by Proposition \ref{uniform constant}, $\mu(V)\leq\lambda(V_{\frac \pi 2 - \alpha})$ for some constant $\alpha$. Since $\lambda(V_{\frac \pi 2 - \alpha})\leq |N_G(V)|$, we obtain that $|V|\leq |N_G(V)|$.

We just proved the equivalence of the two conditions. Given the weak Aleksandrov assumption, Hall's marriage Theorem tells us that the graph formed from its vertices has a matching. That is, there exists a $\sigma$ permutation such that $v_ju_{\sigma(j)}>0$ for all $j$. Thus, $\mathbb{F}_p\neq\emps$.
\end{proof}
\begin{rmk}
In some sense, all possible Aleksandrov relations, even for the non-discrete problem, are related to Hall's Marriage Theorem. It is interesting to note that the theorems related to the Discrete Gauss Image problem provided in the next sections actually prove Hall's marriage theorem for specific bipartite graphs $G(\mu,\lambda)$. If we show the existence of a convex body $K\in\kno$ such that $\mu=\lambda(K,\cdot)$, then the assignment function with respect to this body gives a perfect matching. In some sense, the results of the next section present a functional approach to Hall's marriage theorem. 
\end{rmk}

\begin{defi}Given discrete measures $\lambda$ and $\mu$, we define \textit{the assignment functional} $A(\cdot):\mathbb{F_{\mu,\lambda}}\rightarrow \mathbb R\cup \{-\infty\}$ by 
\begin{equation}
   A(f)=\sum_{j=1}^k\log{u_j v_{f(j)}} 
\end{equation}
for $f\in \mathbb{F_{\mu,\lambda}}$, where $\log(x) = -\infty$ for $x \leq 0$. 
\end{defi}

Note that $f\in\mathbb{F}_p$ if and only if $A(f)>-\infty$. The $-\infty$ in the definition does not create a problem since $u_jv_{f(j)}\leq 1$ for any unit vectors. That is, we never sum $-\infty$ with $+\infty$. In the next section, we will show that as long as the assignment functional is uniquely maximized on the set $\mathbb{F}_{\mu,\lambda}$, and $\mathbb F_p\neq\emps$, there exists a polytope $P\in\po_\mu$ satisfying $\mu=\lambda(P,\cdot)$. On the other hand, if the maximizer is non-unique, then we show that the solution does not exist.

\begin{com}\label{Remark for general problem} 

We would like to comment on the remark made in the Introduction, see Subsection \ref{necc and suff disuss}, about restricting the problem to an equal-weight measures $\lambda$, where we have claimed that equal-weight discrete measure $\lambda$ preserves all variety of the geometric problem, and any non equal-weight discrete measure $\lambda$ can be reduced to a series of equal weight problems. 

Suppose for a given discrete measures $\mu$ and $\lambda$, we have $\mathbb{F_{\mu,\lambda}}\neq\emps$. Let $f\in\mathbb{F}_{\mu,\lambda}$. Then we can define new discrete measures $\lambda'$ and $\mu'$ to be 
\begin{equation}\label{Reduction of Problem}
\begin{split}
  \lambda'&:=\sum\limits_{j=1}^k \delta_{u_j} \\ \mu'&=\sum\limits_{i=1}^m\mu'_i\delta_{v_i}:=\sum\limits_{i=1}^m|f^{-1}(i)|\delta_{v_i},
  \end{split}
  \end{equation}
where $|f^{-1}(i)|$ denotes the number of elements mapping to $i$. It is easy to see that $f\in\mathbb{F}_{\mu',\lambda'}$. If there exists a $K\in\kno$ solving $\mu'=\lambda'(K,\cdot)$ such that $u_j\in \interior{\balpha_K(v_{f(j)})}$, then $K$ automatically solves the original problem $\mu=\lambda(K,\cdot)$ and vice versa. Therefore, $f$ is a solution function for discrete $\mu,\lambda$ problem if and only if $f$ is a solution function for discrete $\mu',\lambda'$ problem where $\lambda'$ is equal-weight. Therefore, any non equal-weight $\lambda$ problem can be residues to several equal-weight problems (one for each $f\in\mathbb{F}_{\mu,\lambda}$), which shows that equal-weight problem takes into considerations all possible geometric structures.  
\end{com}

\section{Existence and Uniqueness}\label{Main section}

In this section, we will prove Theorem \ref{main}, which presents an equivalence between the Discrete Gauss Image problem and The Assignment Problem. We note that for convenience, we will assume that $\mu$ is not concentrated on a closed hemisphere. It seems that nothing prevents similar results to hold without this restriction. Roughly, this happens because if $\mu$ is not concentrated on a closed hemisphere, then by considering some body $K$, one can start to decrease this body in radial directions $\rho_K(u)$ where $uv_i\leq 0$ increasing each $\balpha_K(v_i)$. However, we do not proceed in this direction, since by assuming that $\mu$ is not concentrated on a closed hemisphere, we can speak of a class of polytopes $P_{\mu}$ that has a natural resemblance with the discrete version of classical Minkowski problem.

\begin{thm}\label{main}
Let $\lambda$ be a discrete equal-weight measure and $\mu$ be a discrete measure. Suppose they are weak Aleksandrov related and $\mu$ is not concentrated on a closed hemisphere. Then $\mathbb{F}_p$ is nonempty. Moreover, $f\in\mathbb{F}$ is a solution function if and only if it is the unique maximizer of the assignment functional. In other words,

\begin{itemize}
  \item The assignment functional, $A(\cdot)$, is maximized at exactly one $f\in\mathbb{F}$. For this $f$, there exists a polytope $P\in\po_\mu$ such that $\lambda(P,\cdot)=\mu$ and $u_j\in\interior{\balpha_P(v_{f(j)})}$. 
  \item Or $A(\cdot)$ is maximized by more than one $f\in\mathbb{F}$, in which case there is no convex body $K\in\kno$ such that $\lambda(K,\cdot)=\mu$.
\end{itemize}

\end{thm}

Recall the functional $\Phi(K,\mu,\lambda)$ defined in \eqref{GIP functional}. We start with some preliminary propositions.

\begin{prop}\label{est}
	Let $\lambda$ and $\mu$ be discrete measures. Suppose they are weak Aleksandrov related. Suppose also that $\mu$ is not concentrated on a closed hemisphere. Then, for any assignment function $f\in\mathbb{F}$ and any $K\in\kno$, we have:
	\begin{equation}
		 \Phi(K,\mu,\lambda)\leq -A(f).
	\end{equation}
   
\end{prop}
\begin{proof}
	Given any $K\in\kno$, let $P\in\po_\mu$ be a convex polytope with vertices $r_P(v_i)=r_K(v_i)$. Then $P\subset K$, and hence $h_P\leq h_k$, which implies that $\int\log\rho_{K^*}d\lambda\leq\int\log\rho_{P^*}d\lambda$. At the same time, $\int\log\rho_Kd\mu=\int\log\rho_Pd\mu$ since we preserved the radial distance at the point masses of $\mu$. Hence, $\Phi(K,\mu,\lambda)\leq\Phi(P,\mu,\lambda)$. Recall that we can write $P$ as $P^*=\bigcap_{i=1}^mH^-(1/\alpha_i,v_i)$, where $\alpha_i=\rho_K(v_i)$. 
	
		Note that $\frac{1}{\alpha_iv_iu_j}$ is the distance from the center to the intersection between a ray in the $u_j$ direction starting at the center and the hyperplane $H(1/\alpha_i,v_i)$. From this, we obtain that \begin{equation}\label{eq5.1}
  \log\rho_{P^*}(u_j)=\min_{i=1}^m\log(\frac{1}{\alpha_iv_iu_j}),
\end{equation}
where we assume that $\log(\frac{1}{x})=\infty$ if $x<0$. This relates to the fact that a ray in the $u_j$ direction from the center never intersects the hyperplane $H(1/\alpha_i,v_i)$. Note that by Proposition \ref{Aleks Relation}, \begin{equation}
  1=\mu(\bla u_j\bra)\leq \lambda(\bla u_j\bra_{\frac \pi 2}).
\end{equation}
So, there always exists some $v_i$ such that $u_jv_i>0$, and the right hand-side of equation \eqref{eq5.1} is not $\infty$. Using \eqref{GIP functional}, we write: 
		
		\begin{equation}\label{discrete eq}
  \Phi(P,\mu,\lambda)=\sum_{i=1}^m\mu_i\log\alpha_i+\sum_{j=1}^k\lambda_j\min_{i=1}^m\log(\frac{1}{\alpha_iv_iu_j}).
\end{equation}
So in particular, again assuming $\log(\frac{1}{x})=\infty$ if $x<0$, given any assignment function $f$ we have:
	
	\begin{equation}\label{est 5.5}
  \Phi(P,\mu,\lambda)\leq\sum_{i=1}^m\mu_i\log\alpha_i+\sum_{j=1}^k\lambda_j\log(\frac{1}{\alpha_{f(j)}v_{f(j)}u_j}).
\end{equation}

By definition of the assignment function, \begin{equation}
  \sum\limits_{j\in f^{-1}(i)} \lambda_j = \mu_i.
\end{equation}
 Therefore, looking back at \eqref{est 5.5}, we obtain that $\alpha_i$ on the right side cancels out, which makes it equal to $-A(f)$. Combining the last inequality with $\Phi(K,\mu,\lambda)\leq\Phi(P,\mu,\lambda)$, we prove the claim.
\end{proof}

We easily obtain the following uniqueness result as a corollary of the above argument:

\begin{prop}[Uniqueness]\label{uniq discrete}
If $g\in\mathbb{F}_{\mu,\lambda}$ is a solution function where $\lambda$ is an equal-weight measure, and $\mu$ is not concentrated on a closed hemisphere, then $g$ is the maximizer of the assignment functional.	
\end{prop}
\begin{proof}
Since $g\in\mathbb{F}$ is a solution function, let $P\in\po_\mu$ be a polytope solution to the Gauss Image problem such that $\mu=\lambda(P,\cdot)$	and $u_j\in \interior{\balpha_K(v_{g(j)})}$, which is guaranteed by Proposition \ref{exist polytope}. Since we additionally have:  \begin{equation}
  \sum\limits_{j\in f^{-1}(i)} \lambda_j = \mu_i 
\end{equation}
and since $\lambda_j=1$ by the equal-weight assumption, Equation (\ref{discrete eq}) in the above Proposition \ref{est} gives us that $\Phi(P,\mu,\lambda)=-A(g)$. Combining this with the result of the same proposition, we obtain that $-A(g)\leq-A(f)$ for any $f\in\mathbb{F}$.

\end{proof}

Define $d\lambda_\eps:=\sum_{j=1}^k\phi_{\eps,j}d\bar\lambda$, where $\bar\lambda$ is the Lebesgue measure on the sphere and $\phi_{\eps,i}$ is a bump function taking only two values ($0$ and $\eps^{-1}$) with disk support centered at $u_j$ of volume $\eps$.

\begin{lem}\label{unif const for eps}
Let $\mu$ and $\lambda$ be discrete and weak Aleksandrov related measures. Then there exists $\eps'>0$ and $\alpha>0$ such that for all $\eps<\eps'$, we have that for each closed set $\omega\subset\sn$: 
	\begin{equation}
		\mu(\omega)\leq\lambda_\eps(\omega_{\frac\pi2-\alpha}).
	\end{equation}
	We call $\alpha$ a uniform weak Aleksandrov constant for $\lambda_\epsilon$.
\end{lem}
\begin{proof}
Since $\mu$ and $\lambda$ are weak Aleksandrov related discrete measures, Proposition \ref{uniform constant} tells us that there exists an $\alpha'>0$ such that for any closed set $\omega\subset\sn$,
\begin{equation}\label{5.8}
\mu(\omega)\leq\lambda(\omega_{\frac\pi2-\alpha'}).
\end{equation}
Choose $\eps'$ small enough so that $\spt\phi_{\eps',j}\subset {(u_j)}_{\frac{\alpha'}2}$ for any $j$. If $u_j\in \omega_{\frac \pi 2-\alpha'}$, then $u_j\in (v)_{\frac \pi 2-\alpha'}$ for some $v\in \omega$. Thus, ${(u_j)}_{\frac{\alpha'}2}\subset (v)_{\frac \pi 2-\frac {\alpha'} 2}$. Combining everything, we conclude that for any $\eps\leq\eps'$, if $u_j\in \omega_{\frac \pi 2-\alpha'}$, then

\begin{equation}
	\spt\phi_{\eps,j}\subset {(u_j)}_{\frac{\alpha'}2}\subset(v)_{\frac \pi 2-\frac {\alpha'} 2}\subset \omega_{\frac \pi 2-\frac {\alpha'} 2}.
\end{equation}

  Therefore, since all of the original mass of measure $\lambda$ still remains in $\omega_{\frac \pi 2-\frac {\alpha'} 2}$ for $\lambda_\eps$ (with possible additional mass from the other $u_j$), we conclude that for any $\eps\leq\eps'$,

\begin{equation}
  \lambda(\omega_{\frac\pi2-\alpha'})\leq \lambda_\eps(\omega_{\frac\pi2-\frac{\alpha'}2}). 
\end{equation}
Combining this with \eqref{5.8} and letting $\alpha:=\frac{\alpha'}2$, we complete the proof. 
\end{proof}

\begin{lem}\label{delta-relation}\label{eps sol}
	Given a discrete measure $\mu$ not concentrated on a closed hemisphere and a discrete equal-weight $\lambda$, suppose they satisfy the weak Aleksandrov relation. Then, there exists an $\eps'>0$ such that for all $0<\eps<\eps'$, there exists $K_\eps\in\po_\mu$ solving $\mu=\lambda_\eps(K_\eps,\cdot)$ of the following form:
	\begin{equation}\label{form}
  K_\eps:=(\bigcap_{i=1}^mH^-(1/\beta_{\eps,i},v_i))^*,
\end{equation}
 where $\beta_{\eps,i}=\rho_{K_\eps}(v_i)$. For each $\eps<\eps'$, ${\mathfrak R_{K_\eps}}=1$, there exists a uniform constant $C_{\mu,\lambda}$ such that:
 \begin{equation}
  {\mathfrak r_{K_\eps}}>C_{\mu,\lambda}>0.
\end{equation}
\end{lem}
\begin{rmk}
This constant depends on vectors $v_i$ and the uniform weak Aleksandrov constant $\alpha$ from Lemma \ref{unif const for eps}.
\end{rmk}

\begin{proof}
Let $\eps'$ and $\alpha$ be constants given by Lemma \ref{unif const for eps}. If the measure $\mu$ is not concentrated on a closed hemisphere, then $K_\eps$ and $C_{\mu,\lambda}$ exist by Lemma \ref{weak Aleks bound} in the Appendix. The independence of constant $C_{\mu,\lambda}$ is guaranteed by uniformity of $\alpha$ for all $\eps'<\eps $. 
\end{proof}

For the rest of this section, we will work with constants $\eps'$, $\alpha$, $C_{\mu,\lambda}$ and polytopes $K_\eps$ from previous Lemmas. Define $\spt_{i,j,\eps}\subset S^{n-1}$ as the region of intersection between the support of $\phi_{\eps,j}$ and $\balpha_{K_\eps}(v_i)$. We note two identities below. The first comes from the definition of $\lambda_\eps$, and the second comes from $\mu=\lambda_\eps(K_\eps,\cdot)$. 

\begin{equation}\label{sum}
\begin{split}
 \bar\lambda(\sum_{i=1}^m\spt_{i,j,\eps})&=\eps \\
 \bar\lambda(\sum_{j=1}^k\spt_{i,j,\eps})&=\eps\mu_i.
\end{split}
   \end{equation}

Before we proceed with the proof, we provide the general picture. Solving the Discrete Gauss Image problem corresponds to putting point-masses at $u_j$ into corresponding normal cones $v_i$. By smoothing out these point masses, we get a sequence of bodies $K_\eps$ such that each original point mass has split into possibly several normal cones. The idea is that by making $\eps\rightarrow 0$, the supports of $\lambda_\eps$ become smaller. This should force the supports to be fully contained in a normal cone eventually. If this happens for some body $K$, then $\mu=\lambda_\eps(K,\cdot)=\lambda(K,\cdot)$. We will see that this does not always happen.  

Let us compute $\Phi(K_\eps,\mu,\lambda_\eps)$. Note that $\spt_{i,j,\eps}\subset\balpha_{K_\eps}(v_i)\subset (v_i)_{ \arccos C_{\mu,\lambda}}$, by definition and Lemma \ref{bound on normal cone}. Therefore, in particular, all logarithms are defined in the following expression:
\begin{equation}
\begin{split}
  \Phi(K_\eps,\mu,\lambda_\eps) 
  & =\int\log\rho_{K_\eps}d\mu+\int\log\rho_{K_\eps^*}d\lambda_\eps \\
  & =\sum_{i=1}^m\mu_i\log\beta_{\eps,i}+\sum_{i,j=1}^{i=m,j=k}\int_{\spt_{i,j,\eps}}\log{\frac{1}{\beta_{\eps,i}v_iv}}\frac{1}{\eps}d\bar\lambda(v). 
\end{split}
\end{equation}

Inside of the right integral, we can pull $\beta_i$ out. Then summing over $j$ and using the above identities (\ref{sum}), we get that the first sum cancels out. We have

\begin{equation}\label{f1}
  \Phi(K_\eps,\mu,\lambda_\eps)= -\sum_{i,j=1}^{i=m,j=k}\int_{\spt_{i,j,\eps}}\log{(v_iv)}\frac{1}{\eps}d\bar\lambda(v).
\end{equation}

\begin{lem}\label{convergence}
Under the conclusion of Lemma \ref{delta-relation}, there exists a subsequence of $K_\eps$ converging as $\eps  \rightarrow 0$ to some convex body $K\in\po_\mu$ of the form \begin{equation}
  K=(\bigcap_{i=1}^mH^-(1/\alpha_i,v_i))^* ,
\end{equation}
where $\beta_{\eps,i}\rightarrow\alpha_i>0$ along this subsequence. For this body, $\mathfrak R_{K}=1$ and $\mathfrak r_{K}\geq C_{\mu,\lambda}$.
\end{lem}
\begin{proof}

Since for each $\eps<\eps'$, we have that $ \mathfrak R_{K_\eps}=1$ and $\mathfrak r_{K_\eps}>C_{\mu,\lambda}$, the standard compactness arguments apply.

\end{proof}

\begin{lem}\label{sum1}
	For subsequence $K_\eps \rightarrow K$ in Lemma \ref{convergence}, \begin{equation}
  \Phi(K_\eps,\mu,\lambda_\eps)\rightarrow\Phi(K,\mu,\lambda).
\end{equation}
\end{lem}
\begin{proof}
	We will assume the convergence for all sequences $K_\eps\rightarrow K$ to simplify notation. This implies that $\rho_{K_\eps}\rightarrow \rho_K$ and $\rho_{K^*_\eps}\rightarrow \rho_{K^*}$. Therefore, $\int\log\rho_{K_\eps} d\mu \rightarrow \int\log\rho_Kd\mu$. Now for the second summand of $\Phi$, consider the following:
	
	\begin{equation}\label{est1}
  \int\log\rho_{K_\eps^*}d\lambda_\eps-\int\log\rho_{K^*}d\lambda = \int\log\rho_{K_\eps^*}-\log\rho_{K^*}d\lambda_\eps+\int\log\rho_{K^*}(d\lambda_\eps-d\lambda).
\end{equation}

First, we note that $\rho_{K^*_\eps}\geq1$ and $\rho_{K^*}\geq1$. Since for any $a,b$ we have that  $\log a - \log b\leq\frac 1 {min(a,b)}|a-b|$, we obtain $\log\rho_{K_\eps^*}-\log\rho_{K^*}\leq|\rho_{K^*_\eps}- \rho_{K^*}|$. Since $\rho_{K^*_\eps}\rightarrow \rho_{K^*}$ converges uniformly and $\int d\lambda_\eps = k$, we obtain that:

\begin{equation}
	\int\log\rho_{K_\eps^*}-\log\rho_{K^*}d\lambda_\eps\rightarrow 0.
\end{equation}

It is not hard to see that our bump function $\phi_{\eps,i}$ works nicely for limits of continuous functions. That is, $\lambda_\eps$ converges to $\lambda$ weakly as functionals on the space of continuous functions. This implies that the second summand of (\ref{est1}) approaches zero. Combining all of the above implies that $\Phi(K_\eps,\mu,\lambda_\eps)\rightarrow\Phi(K,\mu,\lambda)$.

\end{proof}

\begin{lem}\label{sum2}
	For subsequence $K_\eps \rightarrow K$ in Lemma \ref{convergence}, there exist a subsequence such that for all $i,j$, \begin{equation}
  \int_{\spt_{i,j,\eps}}\frac 1 \eps d\bar\lambda \rightarrow c_{i,j}.
\end{equation} 
Coefficients $c_{i,j}$ form a matrix $[c_{i,j}]$ with $k$ columns and $m$ rows such that each coefficient in the matrix is nonnegative, the sum over any column is equal to 1, and the sum over the $i$-th row is equal to $\mu_i$.
\end{lem}
\begin{rmk}
This corresponds to how much of weight falls into normal cone of vertex. In the special case $\mu_i=1$, for the equal-weights problem, $[c_{i,j}]$ is a doubly stochastic matrix.
\end{rmk}
\begin{proof}
We start with the subsequence from Lemma \ref{convergence}. 
Since \begin{equation}
   0\leq \int_{\spt_{i,j,\eps}}\frac 1 \eps d\bar\lambda \leq 1,
\end{equation}
the standard compactness arguments insure convergence of volumes of supports for some further subsequence. The coefficients of the matrix are clearly nonnegative. Properties of the matrix follow from identities (\ref{sum}).
\end{proof}

From all the above, we obtain the following form for the functional $\Phi(K,\mu,\lambda)$.

\begin{lem}\label{f2}
For $K$ obtained in Lemma \ref{sum2}, we have:
\begin{equation}
  \Phi(K,\mu,\lambda)= -\sum_{i,j=1}^{i=m,j=k}c_{i,j}\log(v_iu_j)>0.
\end{equation}
Note that if $v_iu_j\leq 0$ then $c_{i,j}=0$, we force the term $c_{i,j}\log(v_iu_j)$ to be zero in the above sum.
\end{lem}
\begin{proof}
Recall that from Lemma \ref{convergence}, we have a bound $ \mathfrak r_{K_\eps}>C_{\mu,\lambda}$ and  $\mathfrak R_{K_\eps}=1$. Therefore, $\spt_{i,j,\eps}\subset\balpha_{K_\eps}(v_i)\subset (v_i)_{\arccos{C_{\mu,\lambda}}}$, by Lemma \ref{bound on normal cone}. In particular, this means that $\log(v_iv)$ is continuous and bounded from below on $\spt_{i,j,\eps}=\spt\phi_{\eps,j}\cap\balpha_{K_\eps}(v_i)$, with bound independent of $\eps$. This with Lemma \ref{sum2} and the construction of the bump function implies that
\begin{equation}
  \int_{\spt_{i,j,\eps}}\log{(v_iv)}\frac{1}{\eps}d\bar\lambda(v)\rightarrow c_{i,j}\log(v_iu_j).
\end{equation}
Combining this with Lemma \ref{sum1} and Equation (\ref{f1}), we obtain the stated result.
\end{proof}
 The matrix $[c_{i,j}]$ is known as the solution matrix to transportation problems of transferring desirable material from $m$ sources to $k$ locations. All such matrices ($k$ columns and $m$ rows such that each coefficient in matrix is nonnegative, sum over any column is equal to 1, and sum over the $i$-th row is equal to $\mu_i$) are known as transportation polytopes. For reference, see Section 8.1 in \cite{Brualdi}. It is a convex polytope with known extreme points. The special case of this polytope, when $k=m$ and $\mu_i=1$, is known as the Birkhoff-von Neumann theorem.

\begin{lem}\label{mat}
There is a natural one-to-one correspondence between vertices of transportation polytopes containing matrix $[c_{i,j}]$ and assignment functionals from $\mathbb{F}$. 
\end{lem}
\begin{proof}
Corollary 8.1.4 in \cite{Brualdi} provides a recursive construction for vertices of the transportation polytope. It is not hard to see from this construction that each vertex is an assignment functional, and every assignment functional is a vertex.
\end{proof}
Convexity of the transportation polytope implies the following:
\begin{prop}\label{eq5} There exist $0\leq\theta_f\leq1$ for $f\in\mathbb{F}$, such that $\sum_{f\in\mathbb{F}}\theta_f=1$ and 
	\begin{equation}
  \Phi(K,\mu,\lambda)= -\sum_{f\in\mathbb F}\theta_fA(f).
\end{equation}
Since  $\Phi(K,\mu,\lambda)<\infty$, we obtain that there exists $f\in\mathbb{F}$ such that $u_jv_{f(j)}>0$. In the above sum, if $A(f)=-\infty$, then $\theta_f=0$.
\end{prop}
\begin{proof}
	From Lemma \ref{f2}, $-\Phi(K,\mu,\lambda)$ is the element-wise product of matrices $[c_{i,j}]$ and $[\log(v_iu_j)]$. By convexity of the transportation polytope and Lemma \ref{mat}, we obtain the above result.
\end{proof}

Now the Theorem \ref{main} follows:
\begin{proof}[\textbf{Theorem \ref{main}. First part. Unique maximizer.} Proof.] We have already seen that assignment functional can only be a solution if it is a maximizer, by Proposition \ref{uniq discrete}. Suppose the assignment functional, $A(g)$, is maximized at exactly one $g\in\mathbb{F}$. Combining Proposition \ref{eq5} and Proposition \ref{est} we obtain that for the constructed body $K$ in Lemma \ref{sum2}, \begin{equation}\label{eq15}
  \Phi(K,\mu,\lambda)= -\sum_{f\in\mathbb F}\theta_fA(f)\geq -A(g)\geq \Phi(K,\mu,\lambda).
\end{equation}
 Therefore, $\theta_g=1$, and thus $[c_{i,j}]$ is a vertex of the transport polytope. In particular, this means that from Lemma \ref{sum2},

  \begin{equation}\label{eq10}
  \begin{split}
  \int_{\spt_{i,j,\eps}} \frac 1 \eps &d\bar\lambda \rightarrow 0 \text{ if } g(j)\neq i\\
  \int_{\spt_{i,j,\eps}} \frac 1 \eps &d\bar\lambda \rightarrow 1 \text{ if } g(j)= i.
\end{split}
\end{equation}
Consider some $K_\eps$. We claim that if $\int_{\spt_{i,j,\eps}} \frac 1 \eps  d\bar\lambda>\frac{1}{2}$, then $\lambda_j\in \interior{\balpha_{K_\eps}(v_i)}$. This is because $\spt_{i,j,\eps}$ is an intersection of the convex normal cone at a point $v_i$ with a spherical disk of volume $\eps$. If the intersection of such sets contains at least $\frac{1}{2}$ of the volume of disk, the center of the disk lies in the interior of the normal cone. So \eqref{eq10} implies that for small enough $\eps$, $\lambda_j\in \interior{\balpha_{K_\eps}(v_i)}$, when $g(j)= i$. Therefore, $\mu=\lambda(K_\eps,\cdot)$. In particular, $g$ is the solution function.
\end{proof}
\begin{rmk}
While we proved that for small enough $\eps$ over some subsequence $K_\eps$ is a solution, we can not guarantee that $K$ is a solution. The problem is that vectors $u_j$ may lie on the boundary of $\balpha_{K}(v_i)$ in general (even if  $\Phi(K,\mu,\lambda)=-A(g)$).
\end{rmk}

\begin{prop}\label{pertub}
	For any discrete $\lambda$ and discrete $\mu$ not contained in a closed hemisphere, polytope solutions to the Discrete Gauss Image problem from the set $\po_\mu$: \begin{equation}
  (\bigcap_{i=1}^mH^-(1/\alpha_i,v_i))^* 
\end{equation}
form an open set in terms of $(\alpha_1,\ldots\alpha_m)\in\mathbb{R}^m$. 
\end{prop}
\begin{proof}
Polytope $K\in\kno$ is a solution if and only if $u_j\in\mathring{\balpha_K(v_{f_K(j)})}$. The latter allows for small perturbations of $(\alpha_1,\ldots\alpha_m)$. 
\end{proof}

\begin{proof}[\textbf{Theorem \ref{main}. Second part. Nonunique maximizer.} Proof.]
	Given $\mu$ and $\lambda$, suppose there exists a solution $K$ such that $\mu=\lambda(K,\cdot)$. Let $P$ be a polytope solution guaranteed by Proposition \ref{exist polytope}. Note that the assignment function $f$ with respect to the body $P$ is a maximizer of $A(\cdot)$ by Proposition \ref{uniq discrete}. Since $u_j\in\mathring{\balpha_P(v_{f(j)})}$, for any small perturbation of $u_j$, $P$ will still be a solution to the Gauss Image problem for small perturbations of $\lambda$, by Proposition \ref{pertub}. So, in particular, $f$ will still be maximizer for small perturbations of $u_j$, again by Proposition \ref{uniq discrete}. We will arrive at a contradiction by showing that if we have more than one maximizer, then we can always perturb $u_j$ small enough so that $f$ is no longer a maximizer of $A(\cdot)$. 
	
	If $f,g$ are two maximizers of $A(\cdot)$, without loss of generality, suppose $f(1)\neq g(1)$. We have that 
	\begin{equation}
  A(f)=\sum_{j=1}^k\log{u_j v_{f(j)}}=\sum_{j=1}^k\log{u_j v_{g(j)}}=A(g) .
\end{equation}
Since $f(1)\neq g(1)$, we can perturb $u_1$ sufficiently small to decrease $\log(\frac{u_1v_{f(1)}}{u_1v_{g(1)}})$, since $u_1v_{f(1)}>0$ and $u_1v_{g(1)}>0$. Thus, the above equality changes to inequality since all other terms are held constant. Therefore, for this perturbation, $f$ is no longer a maximizer of $A(\cdot)$. 
\end{proof}

\section{Alterntative Approach}\label{Alternative Approach}

Since we are working with polytopes, one might wonder whether there exists a more algebraic approach that does not rely on the smoothing of measures. Indeed, such an approach exists. The advantage of this direction is a simple algebraic combinatorial system that does not rely on the smoothing argument. We will not pursue this approach in its full generality, as it does not yield a better result. Yet, we believe this technique is important to be stated.

Rockafellar first characterized the set of subdifferentials of convex functions. He proved that a set is cyclically monotone if and only if it is a subgradient of a convex funtion (see \cite{Rock}). Later, a generalization of the cyclic monotonicity was introduced by Rochet and R\"{u}schendorf (see \cite{rush,Rochet}), in the context of mathematical economics and mass transport. More recently, Artstein-Avidan, Sadovsky, and Wyczesany, in \cite{Kasia}, generalized cyclic monotoic sets even further to what is called c-path-boundedess. The approach we will present is based on similar ideas. 

Here, we consider the most simple case of the Discrete Gauss Image problem: the same number of equal mass atoms for $\lambda$ and $\mu$. Let \begin{equation}
  \lambda=\sum\limits_{j=1}^m \delta_{u_j} \text{ and } \mu=\sum\limits_{j=1}^m \delta_{v_j}.
\end{equation}

\begin{lem}\label{convenience lem}
	Suppose measures $\mu$ and $\lambda$ are discrete, equal-weight, and weak Aleksandrov related. Suppose $\mu$ is not concentrated on a closed hemisphere. Then we can reorder indices so that the identity permutation maximizes $A(\cdot)$. In this case, $u_jv_j>0$ for all $j$.
\end{lem}
\begin{proof}
This immediately follows from Lemma \ref{Hall}.	
\end{proof}

For the rest of this section, we will fix the indexing so that the identity permutation maximizes $A(\cdot)$. This is guaranteed by Lemma \ref{convenience lem}. The next Lemma provides the algebraic system that is equivalent to the existence of the solution to the Gauss Image problem. 

\begin{lem}\label{equation comb}
Suppose measures $\mu$ and $\lambda$ are discrete, equal-weight, and weak Aleksandrov related. Suppose $\mu$ is not concentrated on a closed hemisphere. Consider any polytope $P\in\po_\mu$ \begin{equation}
  P=(\bigcap_{i=1}^mH^-(1/\alpha_i,v_i))^* ,
\end{equation}
where $\alpha_i>0$. Consider the following system of equations
\begin{equation}\label{equation system}
  a_{j,i}<x_j-x_i \text{ for } i\neq j\in\{1\ldots m\} \\,
\end{equation}
where $a_{j,i}=\log\frac{u_jv_i}{u_jv_j}$ if $\frac{u_jv_i}{u_jv_j}>0$, and  $a_{i,j}=-\infty$ otherwise. We define $a_{i,i}=0$ for convenience. Then $\mu=\lambda(P,\cdot)$ if and only if $x_i=\log(\alpha_i)$ solves this system of equations \eqref{equation system}.
\end{lem}
\begin{proof}
Fix $P$, and consider its dual $P^*$. Since $\frac{1}{\alpha_iv_iu_j}$ is the distance from the center to the intersection between a ray in the $u_j$ direction and the hyperplane $H(1/\alpha_i,v_i)$, we obtain that: \begin{equation}
  \log\rho_{P^*}(u_j)=\min_{i=1}^m\log(\frac{1}{\alpha_iv_iu_j}).
\end{equation}
In the last equation, we assume that $\log(\frac{1}{x})=+\infty$ if $x<0$. On the other hand, $u_j\in\interior{\alpha_P(v_j)}$ if and only if the ray in the $u_j$ direction intersects $H(1/\alpha_j,v_j)$ strictly before it intersects any other $H(1/\alpha_i,v_i)$ for $j\neq i$. Therefore, $\mu=\lambda(P,\cdot)$ if and only if for all $j\neq i$, the following strict inequality holds:
\begin{equation}
  \log(\frac{1}{\alpha_jv_ju_j})<\log(\frac{1}{\alpha_iv_iu_j}),
\end{equation}
again assuming that the right side is $+\infty$ if $\alpha_iv_iu_j<0$. Recalling that $\alpha_j,\alpha_i,v_ju_j>0$ and assuming that  the $\log$ of a non-positive value is $-\infty$, the above is equivalent to:
\begin{equation}
\log\frac{u_jv_i}{u_jv_j}<\log\alpha_j-\log\alpha_i.	
\end{equation}
Defining $x_k$ to be $\log\alpha_k$, we arrive at the conclusion.
\end{proof}

The conclusion to the following lemma is the analogue to the cyclic monotonicity condition.
It gives us a natural condition on the coefficients in the above lemma.

\begin{lem}\label{condition on system}
Suppose measures $\mu$ and $\lambda$ are discrete, equal-weight, and weak Aleksandrov related. Suppose $\mu$ is not concentrated on a closed hemisphere. Suppose the identity permutation maximizes the assignment functional $A(\cdot)$. Then the maximizer is unique if and only if for any non-trivial permutation $\sigma$ on $\{1\ldots m\}$, \begin{equation}
  a_{\sigma}:=\sum_{i=1}^{m}a_{i,\sigma(i)}<0,
\end{equation}
where coefficients $a_{i,j}$ are given by Lemma \ref{equation comb}.
\end{lem}
\begin{proof}
Let Id denote the identity permutation. Given any nontrivial permutation $\sigma$, \begin{equation}
\begin{split}
  A(\sigma)<A(\text{Id}) \Leftrightarrow \\
  \sum_{j=1}^m \log(u_jv_{\sigma(j)})< \sum_{j=1}^m \log(u_jv_j) \Leftrightarrow \\
  a_\sigma= \sum_{j=1}^{m}a_{j,\sigma(j)}= \sum_{j=1}^m \log(\frac {u_jv_{\sigma(j)}} {u_jv_j} )<0,
  \end{split}
\end{equation}
again assuming that $\log x = -\infty$ if $x<0$ or the expression inside is undefined.

 \end{proof}

We will show that the system established in Lemma \ref{equation comb} has a solution if the coefficients $a_{i,j}$ satisfy the condition of Lemma \ref{condition on system}. This provides an alternative proof for the Discrete Equal-Weight Gauss Image problem. The system with the non-strict inequality already arose in the work of K. Wyczesany \cite{Wyczesany}. The proof is based on induction and Helly's theorem. We refer the reader to \cite{Wyczesany} for details of the proof.

\begin{lem}[K. Wyczesany, Lemma 4.2.8 in \cite{Wyczesany}]\label{Kasia}
Let $\alpha_{i,j}\geq -\infty$ for $i\neq j \in \{1\ldots m\}$, and let $a_{i,j}=0$ for $i=j$. Then the following system of equations 	\begin{equation}
  a_{j,i}\leq x_j-x_i \text{ for } i\neq j\in\{1\ldots m\} \\
\end{equation}
has a real solution if and only if for any permutation $\sigma$,
\begin{equation}
	\sum_{i=1}^{m}a_{i,\sigma(i)}\leq0.
\end{equation}
\end{lem}

It is not hard to see that the above lemma provides us with a similar lemma for the strict system:

\begin{lem}\label{system strict}
Let $\alpha_{i,j}\geq -\infty$ for $i\neq j \in \{1\ldots m\}$ and let $a_{i,j}=0$ for $i=j$. Then the following system of equations: 	\begin{equation}
  a_{j,i}<x_j-x_i \text{ for } i\neq j\in\{1\ldots m\} \\
\end{equation}
has a real solution if and only if for any non trivial permutation $\sigma$,
\begin{equation}\label{condition on a}
  a_{\sigma}:=\sum_{i=1}^{m}a_{i,\sigma(i)}<0.
\end{equation}

\end{lem}

\begin{rmk}
One can prove this directly, with a similar approach to the proof Lemma \ref{Kasia}.
\end{rmk}
\begin{proof}
Assume the system of equations has a solution. Then for any non-identical permutation $\sigma$,

\begin{equation}
\begin{split}
  &a_{\sigma}=\sum_{i=1}^{m}a_{i,\sigma(i)}=\sum_{\sigma(i)\neq i}^{m}a_{i,\sigma(i)}+\sum_{\sigma(i)= i}^{m}a_{i,i}=\\&\sum_{\sigma(i)\neq i}^{m}a_{i,\sigma(i)}<\sum_{\sigma(i)\neq i}^{m}(x_i-x_{\sigma(i)})=0,
  \end{split}
  \end{equation}
where the last equality follows by simply opening the sum and canceling terms.

For the other direction, assume the conditions on coefficients $\alpha_{i,j}$. Then for any $\sigma$ not equal to the identity permutation,
\begin{equation}
	\sum_{i=1}^{m}a_{i,\sigma(i)}<0.
\end{equation}
Because there are finitely many permutations, we can choose uniform $\eps$ such that for each $\sigma$ not equal to the identity permutation,
\begin{equation}
	\sum_{i=1}^{m}a_{i,\sigma(i)}+\eps<0.
\end{equation}
Choose new coefficients $\bar a_{i,j}=a_{i,j}+\frac \eps m$ if $i\neq j$, and $\bar a_{i,j}=a_{i,j}=0$ otherwise. These coefficients satisfy the condition of Lemma \ref{Kasia}. Hence, there is a solution to the following system of equations: 
\begin{equation}
  \bar a_{j,i}\leq x_j-x_i \text{ for } i\neq j\in\{1\ldots m\} \\,
\end{equation}
since $\alpha_{i,j}<\bar\alpha_{i,j}$ for $i\neq j$. The same solution solves the original system. 
\end{proof}

Now we give an alternative proof of the discrete equal-weight problem.

\begin{thm}
Suppose $\mu$ and $\lambda$ are discrete equal-weight measures that satisfy the  weak Aleksandrov inequality. Suppose $\mu$ is not concentrated on a closed hemisphere. Then $A(\cdot)$ is maximized by exactly one $\sigma\in\mathbb F$ if and only if  there exists a polytope $P$ with vertices $r_P(v_i)$ solving $\mu=\lambda(P,\cdot)$.
\end{thm}
\begin{proof}
	By Lemma \ref{convenience lem}, one can order the coefficients so that the identity permutation maximizes the functional $A(\cdot)$. Then by Lemma \ref{equation comb}, there exists a $P$ solving $\mu=\lambda(P,\cdot)$ if and only if there is a solution to the system of equations \eqref{equation system}. By Lemma \ref{system strict}, the system of equations \eqref{equation system} has a solution if and only if condition \eqref{condition on a} is satisfied for $\alpha_{i,j}$. The condition \eqref{condition on a} is equivalent to the unique maximization of the assignment functional, by Lemma \ref{condition on system}.
\end{proof}

\section{The Assignment Problem from Geometric Point of View}\label{The Assignment Problem from Geometric Point of View}

As was shown by Theorem \ref{main} the question of the existence of the solution to the Discrete Gauss Image problem is equivalent to the uniqueness of the maximizer for the assignment functional. Let us now analyze this question. In this section, we will provide a geometric condition that insures that the maximizer is unique, and, hence, there exists a solution to the Discrete Gauss Image Problem. This condition is equivalent to the statement that if $\sigma_1\neq \sigma_2 \in \mathbb F_{\mu,\lambda}$ then $A(\sigma_1)=A(\sigma_2)$, which forces $A(\cdot)$ to be uniquely maximized. This can be seen to be not far away from necessary, yet for the true necessary condition, one has to restrict the class of permutations $\mathbb F_{\mu,\lambda}$ to consider. In the next section, we are going to analyze the uniqueness of the maximizer from the generic point of few.  

It turns out that the question about the uniqueness of maximizers of the assignment functional for the dimension $n=2$ is very easy to understand. Consider the following example. Let $P$ be any regular convex polytope with the center in its interior. Choose set $\{v_1\ldots v_m\}$ of distinct unit vectors such that $r_P(v_i)$ are all vertices of $P$. Since $n=2$ we can also ensure that vector subscript notation is clock-wise ordered. Now let the set $\{u_1\ldots u_m\}$ be normals to the facets of $P$ such that $u_j\bot [r_P(v_j),r_P(v_{j+1})]$ for $j<m$ and $u_m\bot [r_P(v_m),r_P(v_{1})]$. Let \begin{equation}
  \lambda=\sum\limits_{j=1}^m \delta_{u_j} \text{ and } \mu=\sum\limits_{j=1}^m \delta_{v_j}.
\end{equation}

These two measures are weak Aleksandrov related and $\mu$ is not concentrated on a closed hemisphere, hence satisfying conditions of Theorem \ref{main}. We will now prove that $A(\cdot)$ is maximized at exactly two permutations from $\mathbb{F}_{\mu,\lambda}$: the identity permutation $\sigma_{Id}$ and a cycle $\sigma_{s}=(1\ldots m)$. First, a similar argument to Proposition \ref{necessary weak} shows that $\mu$ and $\lambda$ are weak Aleksandrov related. Clearly, $\mu$ is not concentrated on a closed hemisphere. By a simple computation of $\Phi(P,\mu,\lambda)$ analogues to derivation of Proposition \ref{est} and equations \eqref{discrete eq} and \eqref{est 5.5} one can show that for any permutation $\sigma\neq \sigma_s$ and $\sigma \neq \sigma_{Id}$:

	\begin{equation}
		 -A(\sigma_{id})=-A(\sigma_{s})=\Phi(P,\mu,\lambda)<-A(\sigma)
	\end{equation}
	Combining all these Theorem \ref{main} gives us that there doesn't exist a solution $K\in \kno$ such that $\mu=\lambda(K,\cdot)$. It is interesting to note that the solution $K_\eps$ for $\lambda_\eps$,$\mu$-problem from Lemma \ref{convergence} converge to $P$, so in particular, the method established in the Existence and Uniqueness section of this work identifies the body $P$ from which measures were constructed.

This example is the essence of the Assignment Problem. In dimension two if one starts with discrete equal-weight $\mu$ and $\lambda$ satisfying weak Aleksandrov condition such that $\mu$ is not concentrated on a closed hemisphere we either obtain that there exist polytope $P\in\po_\mu$ such that $\mu=\lambda(P,\cdot)$ or there was a polytope $P\in\po_\mu$ such that $\{u_1\ldots u_m\}$ where normals to its facets. The answer depends on the uniqueness of the maximizer. Moreover, the failure of uniqueness precisely determines $P\in\po_\mu$ (up to scaling) and the assignment of $u_j$ to specific side of $P$ in the sense of  $u_j\bot [r_P(v_i),r_P(v_{i+1})]$. 

Now, let's analyze the higher dimensional picture while additionally dropping the equal-weight assumption for $\mu$. First, we define a geometric condition which is related to the uniqueness of the maximizer. This condition comes from "lifting" the mentioned two-dimensional example into higher dimensions. 

%Describe rigidity structure and how this is related to starting from a polytope. 
%In defenition maybe add doesn't pass through center

\begin{defi}
	Given set of distinct vectors $\{u_1\ldots u_l\}$ and distinct vectors $\{v_1\ldots v_l\}$ with $l\geq 2$ suppose there exist a piecewise linear closed curve $\gamma$ with vertices $\{x_1\ldots x_l\}$ such that 
	\begin{itemize}
		\item $x_i=\lambda_iv_i$ for some $\lambda_i>0$
		\item $u_i \bot [x_i,x_{i+1}]$ for $1\leq i \leq l-1$
		\item $u_l \bot [x_l,x_1]$ 
	\end{itemize}
	The curve $\gamma$ is called an edge-normal loop of two sets. %It is called a proper convex edge-normal loop if $\gamma\subset \partial \text{conv}(\gamma)$  and $0\notin \gamma$.
	%This curve dose not always exist. If it exists it is unique up to dialation. Define convex independence. points $\{x_1\ldots x_l\}$ are convexly independent not enough.  
\end{defi}
\begin{rmk}
	This curve doesn't always exist. If it exists it is unique up to dilation.
	
\end{rmk}
%Note the convex curve that it is a part of the boundary of convex polytope?? so you can write better defenition.

Clearly, any piecewise linear closed curve not passing through the center provides two sets $\{u_1\ldots u_l\}$ and $\{v_1\ldots v_l\}$ to which it is an edge-normal loop. Going back to the example discussed at the beginning of this section, any two dimensional polytope defines an edge-normal loop. The next two propositions establish the relation between the edge-normal loops and the values of the assignment functional for different permutations. 

\begin{prop}\label{edge normal loop}
	Suppose a piecewise linear closed curve $\gamma$ with verticies $\{x_1\ldots x_l\}$ is an edge normal loop of $\{u_1\ldots u_l\}$ and  $\{v_1\ldots v_l\}$. Then \begin{equation}
		A(\sigma_{id})=A(\sigma_{s})
	\end{equation}  
	where $\sigma_s$ is cycle permutation $(1\ldots l)$. 

\end{prop}
\begin{proof}
Let $u_{l+1}=u_1$, $v_{l+1}=v_1$, $x_{l+1}=x_1$ and $\lambda_{l+1}=\lambda_1$ then by defenition:
	\begin{equation}\begin{split}
		A(\sigma_{id})=A(\sigma_{s}) &\Leftrightarrow \\ \sum_{j=1}^l \log(u_jv_{j})= \sum_{j=1}^l \log(u_jv_{j+1}) &
		\end{split}
	\end{equation} 
	If $0\in \gamma$ then $0\in [x_i,x_{i+1}]=[\lambda_iv_i,\lambda_{i+1}v_{i+1}]$ and hence $u_iv_i=0$ and $u_iv_{i+1}=0$ which forces $A(\sigma_{id})=A(\sigma_{s})=-\infty$ by convention and  establishes the above equality. Suppose $0\notin \gamma$, then for all $j$, $u_jv_j>0$ and $u_{j}v_{j+1}>0$. Hence the above equality is equivalent to: 
	\begin{equation}
		\prod_{j=1}^l \frac {u_jv_{j}} {u_jv_{j+1}} =1
	\end{equation}
	Now $u_i\bot [x_i,x_{i+1}]$ is exactly $u_i(x_i-x_{i+1})=0$ which is equivalent to $\frac{\lambda_{i+1}}{\lambda_{i}}=\frac {u_iv_{i}} {u_iv_{i+1}}$. Therefore, the above is equivalent to:
	\begin{equation}
		\prod_{j=1}^l \frac{\lambda_{i+1}}{\lambda_{i}}=1
	\end{equation}
	which holds since $\lambda_{l+1}=\lambda_1$.
	
\end{proof}

We now show the reverse statement.

\begin{prop}\label{edge normal loop reverse}
	Suppose for sets $\{u_1\ldots u_l\}$ and $\{v_1\ldots v_l\}$ 
	\begin{equation}
		A(\sigma_{id})= A(\sigma_{s})>-\infty
	\end{equation}  
	Then there exist an edge-normal loop of $\{u_1\ldots u_l\}$ and $\{v_1\ldots v_l\}$.
\end{prop}
\begin{proof}
Similar to previous proposition, since $A(\sigma_{id})= A(\sigma_{s})>-\infty$ we obtain 

	\begin{equation}\label{super equation}
		\prod_{j=1}^l \frac {u_jv_{j}} {u_jv_{j+1}} =1
	\end{equation}
Let $\lambda_1=1$. Then recursively define $\lambda_i$ for $i\leq l$ using the relation $\frac{\lambda_{i+1}}{\lambda_{i}}=\frac {u_iv_{i}} {u_iv_{i+1}}$. By recursive defenition, we obtain that for all $i<l$, $u_i\bot [x_i,x_{i+1}]$. Then using \eqref{super equation} we obtain:
	\begin{equation}
		\frac{\lambda_l}{\lambda_1}=\prod_{j=1}^{l-1} \frac{\lambda_{i+1}}{\lambda_{i}}=\prod_{j=1}^{l-1} \frac {u_jv_{j}} {u_jv_{j+1}}=\frac{u_lv_1}{u_lv_l}
	\end{equation}
	Therefore we also obtain $u_l\bot [x_l,x_{1}]$.
\end{proof}

	\begin{defi}
		Discrete equal-weight measures $\mu$ and $\lambda$ are called edge-normal loop free if for any $\sigma,\sigma'\in S_m$ and given any $l$ such that $2 \leq l \leq m$, there dose not exist and edge normal loop for  $\{u_{\sigma(1)}\ldots u_{\sigma(l)}\}$ and $\{v_{\sigma'(1)}\ldots v_{\sigma'(l)}\}$
	\end{defi}
	
	Before addressing non-equal weight measure $\mu$ let us first illustrate the relation between edge-normal loop free condition and the uniqueness of the assignment functional. 
	
	\begin{thm}
		Suppose $\mu$ and $\lambda$ are discrete equal-weight measures which are weak Aleksandrov related and such that $\mu$ is not concentrated on a closed hemisphere. Suppose $\mu$ and $\lambda$ are edge-normal loop free. Then there exist a unique $\sigma\in\mathbb F_{\mu,\lambda}$ which maximizes $A(\cdot)$.
	\end{thm}   	
	\begin{proof}
		Since $\mu$ and $\lambda$ are weak Aleksandrov related by Lemma \ref{Hall} the set $\mathbb F_{p,\mu,\lambda}$ is nonempty. Hence, there exist a $\sigma\in \mathbb F_{p,\mu,\lambda}$ such that $A(\sigma)>-\infty.$ By reordering indices we can assume that the maximizers is identity permutation. Now chose any other permutation $\sigma\in F_{p,\mu,\lambda}$. Let $\tau_1\ldots\tau_k=\sigma$ be its decomposition into non-trivial non-intersecting cycles. 
		
		Let $\tau_1=(j_1\ldots j_v)$. Consider ordered subsets $\{u_{j_1}, \ldots u_{j_v}\}$ and $\{v_{j_1} \ldots v_{j_v}\}$. Since $\mu$ and $\lambda$ are edge-normal loop free from, we obtain from Proposition \ref{edge normal loop reverse} that 
		$A(\tau_1')\neq A(\sigma_{id}^1)$ where $\sigma_{id}^1$ is identity permutation on the set $\{j_1\ldots j_v\}$ and $\tau_1'$ is a restriction of $\tau_1$ to set $\{j_1\ldots j_v\}$. If $A(\tau_1')>A(\sigma_{Id}^1)$ then $A(\tau_1)>A(\sigma_{Id})$. Thus, the identity permutation is not a maximizer. Contradiction. Hence, $A(\tau_1')<A(\sigma_{Id}^1)$. Similarly, we obtain for all $1\leq i\leq k$ that $A(\tau_i')<A(\sigma_{Id}^i)$. Therefore, $A(\sigma)<A(\sigma_{Id})$. Since $\sigma$ was an arbitrary permutation we are done.
	\end{proof}
	
	Combining this with Theorem \ref{main} we obtain:
	
	\begin{cor}
		Suppose $\mu$ and $\lambda$ are discrete equal-weight measures which are weak Aleksandrov related and such that $\mu$ is not concentrated on a closed hemisphere. Suppose $\mu$ and $\lambda$ are edge-normal loop free. Then, there exists a polytope $P\in\po_\mu$ such that $\lambda(P,\cdot)=\mu$. 
	\end{cor}

We now show the reverse statement which gives another insight into the geometry behind edge-normal loop condition. 

\begin{prop}
	Pick any $\mu$. Given polytope $P\in\po_\mu$, consider a closed path $\gamma$ of its adjacent vertices using the edges of $P$. Suppose $\{x_1\ldots x_k\}$ are its vertices in order. Let $x_{k+1}=x_k$. Pick $u_i$ such that $[x_i,x_{i+1}]\in H_{P}(u_i)$. Then $\gamma$ is an edge normal loop for sets $\{v_1\ldots v_k\}$ and $\{u_1\ldots u_k\}$. 
\end{prop}
\begin{proof}
	Immediate.
\end{proof}

In particular, this shows that in some sense all "relevant" edge-normal loops come from 1-skeleton of convex polytopes. If the edge-normal loop is not realizable as a part of 1-skeleton of a convex polytope with the center in its interior it doesn't affect the solution of the Gauss Image Problem.
 This also shows that in general normals to edges are quite rigid if there exists an edge-normal loop of them. To formulate this more precisely, consider the following question: Suppose we start with some discrete equal-weight measure $\mu$ and polytope $P\in\po_\mu$. Suppose $\lambda$ is a discrete-equal weight measure. Now let us assume that $P$ almost solves the Gauss Image Problem with $\sigma_{Id}$ assignment, but it happened that some of the vectors $u_i$ are actually on the boundary of the normals cones $u_j\in\balpha_P(v_{j})$. Can we change the polytope a bit, moving its vertices along the rays $v_j$ to obtain $u_j\in\interior{\balpha_P(v_{j})}$, so that everything falls inside of normal cones? If we have only one vector on the boundary, then there is no problem, we can move the corresponding vertex a bit outside to increase the normal cone. What about the general case? Well, we can do this if and only if there are no edge-normal loops.
 
 We now turn to the statement for non-equal weight measure $\mu$.  We will need some basic combinatoric machinery established before we proceed. Unfortunately, we were not able to find a good reference for these results. Given permutation $\sigma$ we can uniquely decompose it into disjoint cycles $\tau_j=(x^j_1\ldots x^j_{s_j})$ where by $s_j$ we denote the length of the cycle. We prove a similar statement for assignment functionals:
 
 \begin{prop}\label{cycle decomposition}
 	Let $\lambda$ be a discrete equal-weight measure and $\mu$ be a discrete measure. Suppose $f,g\in \mathbb{F_{\mu,\lambda}}$. Then, there exist a permutation $\sigma\in S_{k}$ such that $f\circ \sigma = g$. Moreover, there exist a product of disjoint cycles $\tau_1,\ldots, \tau_l\in S_k$, such that for each $\tau_j=(x^j_1\ldots x^j_{s_j})$ function $f$ is injective with restrict to the set $\{x^j_1,\ldots, x^j_{s_j}\}$ and $f\circ \tau_1\ldots\tau_l=g$.
 	
 \end{prop}
 \begin{proof}

 	Since $\lvert f^{-1}(i)\rvert=\lvert g^{-1}(i)\rvert$ given any $i\in {1\ldots m}$ choose a bijective function $h_i$ from set $g^{-1}(i)$ into set $f^{-1}(i)$. Since all functions $h_i$ are bijections with non-intersecting supports and images we can define $\sigma\in\sn$ to be permutations satisfying $\sigma(j)=h_{g(j)}(j)$. Then, $f\circ \sigma (j) = fh_{g(j)}(j)=f(f^{-1}(g(j)))=g(j)$, which ensures the first claim. 
 	
 	Now pick any $\sigma$, satisfying $f\circ \sigma = g$. Decompose $\sigma$ into a product of disjoint cycles $\phi_1\ldots\phi_v$. Suppose, without loss of generality, cycle $\phi_1=(x_1\ldots x_{s(j)})$ does not satisfy the claim. We will show that we can, then, always split it into two more cycles $\omega_1$ and $\omega_2$ of a strictly smaller size such that $f\circ\omega_1\omega_2\phi_2 \ldots\phi_v=g$ and all cycles are disjoint. Then, since, in general, we can only do finitely many splitting, eventually we will have that $f$ is injective with respect to any cycle in the decomposition.  
 	
 	We are only left to prove that if $f$ is not injective with respect to elements permuted by $\phi_1$, then we can split $\phi_1$ into $\omega_1$ and $\omega_2$, so that $f\circ \phi_1=f\circ \omega_1\omega_2$. Given $\phi_1=(x_1\ldots x_{s(1)})$ let $i,j$ be the indices such that $f(x_i)=f(x_j)$. Without loss of generality, assume $i=1$. Then we can define $\omega_1=(x_1\ldots x_{j-1})$ and $\omega_2=(x_j\ldots x_{s(1)})$. Now, $\phi_1$ is equal to $\omega_1\omega_2$ everywhere besides $x_{j-1}$ and $x_{s(1)}$. Yet, $f\circ \phi_1(x_j)=f(x_j)=f(x_1)=f\circ \omega_1(x_{j-1})$ and $f\circ \phi_1(x_{s(1)})=f(x_1)=f(x_j)=f\circ \omega_2(x_{s(1)})$. Therefore, $f\circ \phi_1=f\circ \omega_1\omega_2$. 
 \end{proof}

Now with the combinatorial result in hand, we similarly obtain the previous results for a bigger class of measures.

%multiset permutations? annagrams? 

\begin{defi}\label{Sept 17 Defenition}
	Discrete equal-weight measures $\lambda$ and discrete measure $\mu$ are called edge-normal loop free if given any $l$ such that $2 \leq l \leq m$ and given any $\sigma\in S_k$ and $\sigma'\in S_m$ and given any $l$ such that $2 \leq l \leq m$, there dose not exist and edge normal loop for  $\{u_{\sigma(1)}\ldots u_{\sigma(l)}\}$ and $\{v_{\sigma'(1)}\ldots v_{\sigma'(l)}\}$.
\end{defi}

\begin{thm}\label{main edge normal}
	Let $\lambda$ be a discrete equal-weight measure and $\mu$ be a discrete measure. Suppose they are weak Aleksandrov related, edge-normal loop free and $\mu$ is not concentrated on a closed hemisphere. Then there exist a unique $f\in\mathbb F_{\mu,\lambda}$ which maximizes $A(\cdot)$.
\end{thm}
\begin{proof}
		Since $\mu$ and $\lambda$ are weak Aleksandrov related by Theorem \ref{main} the set $\mathbb F_{p,\mu,\lambda}$ is nonempty. Hence, there exist a $f\in \mathbb F_{p,\mu,\lambda}$ such that $A(f)>-\infty.$. Suppose $f,g\in\mathbb F_{p,\mu,\lambda}$ are maximizers of $A(\cdot)$. By Proposition \ref{cycle decomposition} there exist disjoint cycles $\tau_1,\ldots, \tau_l\in S_k$, such that for each $\tau_j=(x^j_1\ldots x^j_{s_j})$ function $f$ is injective with restrict to the set $\{x^j_1,\ldots, x^j_{s_j}\}$ and $f\circ \tau_1\ldots\tau_l=g$. Note that since $f$ is injective on any cycle, $s_j\leq m$ for any $j$.
		
		Let $f_i$ and $g_i$ for $1\leq i\leq l$ be the restrictions of $f$ and $g$ to the subset $\{x^j_1,\ldots, x^j_{s_j}\}$. Then, 
		\begin{displaymath}
		\begin{split}
  A(f)&=\sum_{i=1}^{l}A(f_i) \\
  A(g)&=\sum_{i=1}^{l}A(g_i)=\sum_{i=1}^{l}A(f_i\circ \tau_i)
		\end{split} 
\end{displaymath}
Notice that $A(f_i\circ \tau_i)=A(g_i)$. If, it happened that for some $i$, $A(f_i\circ \tau_i)>A(f_i)$, then $A(f\circ \tau_i)>A(f)$, and hence $f$ is not the maximizer. Therefore, $A(f_i\circ \tau_i)\leq A(f_i)$. Since $\mu$ and $\lambda$ are edge-normal loop free, sets $\{u_{x_1^j}\ldots u_{x^j_{s_j}}\}$ are $\{v_{x_1^j}\ldots v_{x^j_{s_j}}\}$ are edge-normal loop free, and, hence,  $A(f_i\circ \tau_i)\neq A(f_i)$ for all $i$. Therefore, $A(f_i\circ \tau_i)< A(f_i)$. Thus, $A(g)<A(f)$. Contradiction. 
\end{proof}

Combining this with Theorem \ref{main} we obtain:

	\begin{cor}\label{Proof of main corollary}
Let $\lambda$ be a discrete equal-weight measure and $\mu$ be a discrete measure. Suppose they are weak Aleksandrov related, edge-normal loop free and $\mu$ is not concentrated on a closed hemisphere. Then, there exists a polytope $P\in\po_\mu$ such that $\lambda(P,\cdot)=\mu$. 
	\end{cor}
	
	\begin{proof}
		The proof is an immediate combination of Theorem \ref{main} and Theorem \ref{main edge normal}
	\end{proof}

\section{The Assignment Problem from Generic Point of View}\label{Generic Statements}
We will now investigate The Assignment Problem from generic point of view. First, we will show that the maximizer is unique in the generic sense. Suppose we are given a discrete, equal-weight measure $\lambda$ and a discrete measure $\mu$, such that $\mu(\sn)=\lambda(\sn)$. We define $\mathcal A$ to be the class of all possible pairs of measures $\mu',\lambda'$ with variations of directions of the point masses of $\mu$ and $\lambda$. More formally, we can start with some $m\in\mathbb{N}$ and coefficients $\mu_i\in\mathbb{N}$ and consider all possible pairs of measures ($\mu$ and $\lambda$). We will additionally impose that $\mu$ is a discrete measure with fixed coefficients and $\lambda$ is a discrete, equal-weight measure with fixed $k=\sum_{i}^{m}\mu_i$. We can parameterize this class as a product of $m+k$ spheres, with a sphere for each vector. That is, each $(\mu',\lambda')\in \mathcal A$ has representation $(v'_1,\ldots,v'_m,u'_1,\ldots,u'_k)$. The set $\mathcal A$ naturally inherits the topology from spheres. In $\mathcal A$, we also require that $u_j$ be distinct and $v_i$ be distinct. In terms of parameterization, this constitutes an open subset of the product of spheres.

\begin{defi}
Space $\mathcal A$ constructed above is called the $\mu$ problem space. Sometimes, we write $\mathcal A_{\mu}$ to emphasize the original measure from its construction. Note that it only depends on the dimension $n$, constant $k\in\mathbb N$, and coefficients $\mu_i\in \mathbb N$ for $1\leq i\leq k$.  
\end{defi}

\begin{prop}\label{Aleksandrov open}
	Pairs of measures $(\mu,\lambda)\in \mathcal A$ satisfying the weak Aleksandrov inequality form an open set in the inherited topology from the product of spheres. 
\end{prop}
\begin{proof}
Take some $(v_1,\ldots,v_m,u_1,\ldots,u_k)\sim (\mu,\lambda)\in \mathcal A$ satisfying the weak Aleksandrov relation. Then by Proposition \ref{uniform constant}, there exists a uniform $\alpha>0$ such that for each closed set $\omega\subset\sn$: 
\begin{equation}
\mu(\omega)\leq\lambda(\omega_{\frac\pi2-\alpha}).
\end{equation}

Let $(v'_1,\ldots,v'_m,u'_1,\ldots,u'_k)\sim (\mu',\lambda')\in \mathcal A$ be any pair of measures satisfying $u_ju'_j>\cos(\frac \alpha 4)$ and $v_iv'_i>\cos( \frac \alpha 4)$. Note that all such possible measures form an open ball around $(\mu,\lambda)$ in the product topology. Then for any compact, convex set $\omega\subset\sn$, if $v'_i\in\omega$, then $v_i\in\omega_{\frac \alpha 4}$. Hence,

\begin{equation}
\begin{split}
  \mu'(\omega)=\mu'(\omega\cap\{v'_1,\ldots,v'_m\})\leq& \mu(\omega_{\frac \alpha 4}\cap\{v_1,\ldots,v_m\})\end{split}.
\end{equation}
We can apply Proposition \ref{uniform constant} to $\omega_{\frac \alpha 4}\cap\{v_1,\ldots,v_m\}$ to obtain:

\begin{equation}
  \mu'(\omega)\leq \mu(\omega_{\frac \alpha 4}\cap\{v_1,\ldots,v_m\})\leq \lambda((\omega_{\frac \alpha 4}\cap\{v_1,\ldots,v_m\})_{\frac \pi 2 - \alpha}).
\end{equation}
Clearly, $(\omega_{\frac \alpha 4}\cap\{v_1,\ldots,v_m\})_{\frac \pi 2 - \alpha}\subset \omega_{\frac \pi 2 - \frac 3 4 \alpha}$. Thus,

\begin{equation}
  \mu'(\omega)\leq \lambda(\omega_{\frac \pi 2 - \frac 3 4 \alpha}).
\end{equation}
Now if $u_j\in \omega_{\frac \pi 2 - \frac 3 4 \alpha}$, then $u'_j\in \omega_{\frac \pi 2 - \frac 2 4 \alpha}$. We obtain:

\begin{equation}
  \mu'(\omega)\leq \lambda'(\omega_{\frac \pi 2 - \frac 2 4 \alpha}).
\end{equation}
Therefore, $\mu',\lambda'$ are weak Aleksandrov related. 
\end{proof}

\begin{prop}\label{concentration open}
Pairs of measures $(\mu,\lambda)\in \mathcal A$ such that $\mu$ is not concentrated on a closed hemisphere form an open set. 	
\end{prop}
\begin{proof}
Consider a sequence of $(\mu^n,\lambda^n)\rightarrow (\mu,\lambda)$, where $\mu^n$ are concentrated on a closed hemisphere. For each $n$, there exists a $u_n\in\sn$ such that for all $i$, $u_nv^n_i\leq 0$.	 By compactness, there exists a subsequence such that $u_n\rightarrow u$. Since $v^n_i\rightarrow v_i$, then along the subsequence $u_n v^n_i\rightarrow uv_i$. Hence, for all $i$,  $uv_i\leq 0$. Thus, $\mu_n$ is concentrated on a closed hemisphere. Therefore, the set $(\mu,\lambda)$ where $\mu$ is concentrated on a closed hemisphere is closed, and its complement is an open set. 
\end{proof}

\begin{thm}\label{Genetric Theorem}
Given $\mathcal A$, the pairs of measures  $(\mu,\lambda)\in\mathcal A$ for which there exists a polytope solution to the Discrete Gauss Image problem form a dense open set in an open set of measures $(\mu,\lambda)$ satisfying the weak Aleksandrov relation, where $\mu$ is not concentrated on a closed hemisphere. 
\end{thm}
\begin{proof} 
	Let $\mathcal A_c$ be a subset of $\mathcal A$ consisting of pairs of measures that are weak Aleksandrov related and where $\mu$ is not concentrated on a closed hemisphere. Openness of $\mathcal A_c$ follows from Propositions \ref{concentration open} and \ref{Aleksandrov open}. Define $A'(f):=e^{A(f)}$, where $A'(f)=0$ if  $A(f)=-\infty$. We note that the set $\mathbb F$ of assignment functions remains the same for any pair of measures from $\mathcal A$. For each $f\in\mathbb F$, we also define the set $\mathcal A_f:=\{(\mu,\lambda)\mid \forall j \text{ } u_jv_{f(j)}>0\}$.

	 First, we note that $A'(f)$ is a continuous function on $\mathcal A$ as a function from pairs of measures into $\mathbb R$. Moreover, $A'(f)=A'(g)$ if and only if $A(f)=A(g)$, and $A'(f)>A'(g)$ if and only if $A(f)>A(g)$. Therefore, $A(\cdot)$ is uniquely maximized for a fixed pair of measures if and only if $A'(\cdot)$ is uniquely maximized for the same pair. Recall that by Theorem \ref{main}, to prove the statement of the theorem, it suffices to show that the assignment functional is uniquely maximized on a dense open subset of $\mathcal A_c$.

	Suppose for some pair $(\mu,\lambda)\in\mathcal A_c$, the functional is uniquely maximized. Thus, there exists $f\in\mathbb F$, such that $A(f)>A(g)$ for every other $g\in\mathbb F$. Moreover, $A(f)>-\infty$ by Theorem \ref{main}, since $(\mu,\lambda)\in\mathcal A_c$. Therefore, $A'(f)>A'(g)$ for every other $g\in\mathbb F$, and $A'(f)>0$. Then, since there are finitely many $g\in\mathbb F$, by continuity of $A'(f)$ and $A'(g)$ with respect to pairs of measures, there exists a neighborhood of $(\mu,\lambda)$ such that $A'(f)$ is still a maximizer and $A'(f)>0$. Thus, $A(f)$ is a unique maximizer for a neighborhood of $(\mu,\lambda)$. This shows that the set where the assignment functional is uniquely maximized is open in $\mathcal A$. Since $\mathcal A_c$ is open, this set is also open in $\mathcal A_c$.

	Consider some $\mathcal A_f\cap \mathcal A_g=\varnothing$. We would like to show that $A(f)\neq A(g)$ on a dense open subset of $\mathcal A_f\cap \mathcal A_g$. First of all, $A(f)$ and $A(g)$ are continuous functions on  $\mathcal A_f\cap \mathcal A_g$, and thus $A(f)\neq A(g)$ on an open subset of $\mathcal A_f\cap \mathcal A_g$. Suppose now for some pair $(\mu,\lambda)\in \mathcal A_f\cap \mathcal A_g$, we have $A(f)=A(g)$. Then: 
	\begin{equation}\label{Sept 17.111}
	\begin{split}
 & A(f)=\sum_{j=1}^k\log{u_j v_{f(j)}}=\sum_{j=1}^k\log{u_j v_{g(j)}}=A(g) \Leftrightarrow \\
  &\sum_{j=1}^k\log(\frac {u_j v_{f(j)}} {u_j v_{g(j)}})=0.
  \end{split}
\end{equation}
	 Since $f\neq g$, we can find a $j$ such that $f(j)\neq g(j)$. Since the $v_i$ are distinct, we have that $v_{f(j)}\neq v_{g(j)}$. Thus, there exists a small variation of $u_j$ such that all other terms are held constant while $\log(\frac {u_j v_{f(j)}} {u_j v_{g(j)}})$ changes, which makes $A(f)\neq  A(g)$. This shows that $A(f)\neq A(g)$ is a dense subset of $\mathcal A_f\cap \mathcal A_g$, making it a dense open subset of $\mathcal A_f\cap \mathcal A_g$.

	 Now consider some $(\mu,\lambda)\in\mathcal A_c$ such that the functional is not uniquely maximized. Suppose $\{f_i\}_{i\in I}$ are all proper assignment functions. That is, $A(f_i)>-\infty$. Let $\mathcal A_I:=\bigcap \mathcal A_{f_i}$, which is non-empty open set containing $(\mu,\lambda)$. From the previous argument, for each pair $i_1,i_2\in I$, $A(f_{i_1})\neq A(f_{i_2})$ on a dense open subset of $\mathcal A_{f_{i_1}}\cap \mathcal A_{f_{i_2}}$. Hence, $A(f_{i_1})\neq A(f_{i_2})$ on a dense open subset of $\mathcal A_I$. Since a finite intersection of dense open sets is dense and open, there exists a dense open set $\mathcal B_1\subset A_I$ such that all $A(f_i)$ are distinct. Hence, all $A'(f_i)$ are distinct on $\mathcal B_1$. 
	 
	Since $(\mu,\lambda)\in\mathcal A_c$, by Theorem \ref{main}, there exists an $f\in\mathbb F$ such that $A'(f)>0$. By continuity, there exists a neighborhood $\mathcal B_2$ of $(\mu,\lambda)$ such that $A'(f)>A'(g)$ for each $g\notin\{f_i\}_{i\in I}$, since $A'(g)=0$ at $(\mu,\lambda)$. Combining both sets $\mathcal B_1$ and $\mathcal B_1$, we see that for each $(\mu,\lambda)\in \mathcal B:=\mathcal B_1\cap \mathcal B_2$, the functional is uniquely maximized. Moreover, $\mathcal B$ is a dense open set of some neighborhood of $(\mu,\lambda)$. Therefore, there exists a pair of measures with a unique maximizer sufficiently close to $(\mu,\lambda)$. This shows that the set where assignment functional is uniquely maximized is dense in $\mathcal A_c$.

\end{proof}	
\begin{rmk}
Instead of the above generic formulation, we can write generic property with the respect to the Zariski topology. Since for each $f,g\in \mathbb F$, the condition \begin{equation}
  A(f)-A(g)=0
\end{equation}
are zeros of an algebraic function, it suffices to show that the set \begin{equation}
  \{(\mu,\lambda)\mid \forall f,g\text{ }A(f)\neq A(g)\}
\end{equation}
is nonempty. We will still need to separately require the non-concentration condition on $\mu$ and the weak Aleksandrov condition. 
\end{rmk}

Finally, we will give some examples of pairs of measures for which there does not exist a unique maximizer. 

\begin{ex}
	Let $n=3$ and choose some constant $l>3$. Pick any small regular spherical polygon with $2l$ vertices contained in a closed hemisphere. (That is, a polygon on a sphere with edges being great-circle arcs, and same angles between planes defined by the consecutive great-circle arcs.) Iteratively label vertices by $u_1,v_1,u_2,v_2\ldots v_l$. This regular spherical polygon naturally defines a two-dimensional polygon in $\mathbb R^3$ with the same vertices. Let $n$ be a unit normal. Choose $v_{l+1}=u_{l+1}=-n$. Let $\mu$, $\lambda$ be equal-weight discrete measures from these vectors. First of all, the measure $\mu$ is not concentrated on a closed hemisphere. It is also not hard to convince oneself that, by choosing the polygon to be small enough, one can ensure that there are exactly two maximizers of the assignment functional $f$ and $g$, where $f$ is defined by: \begin{equation}
  f(j)=j 
\end{equation}
and $g$ is:
\begin{equation}
  g(j) =
\left\{
	\begin{array}{ll}
		j-1  & \mbox{if } 1< j\leq l \\
		l & \mbox{if } j= 1 \\
		l+1 & \mbox{if } j=l+1.
	\end{array}
\right.
\end{equation}
It is also not hard to see that sets $\{u_1\ldots u_l\}$ and $\{v_1\ldots v_m\}$ have an edge normal loop which is given by piciewise linear path connecting points $x_i=v_i$.

Consider a small variation of vector $u_2$. As long as $\frac{A(f)}{A(g)}=1$, the functional is not uniquely maximized. Holding all other vectors fixed, we can vary $u_2$ as a unit vector or as a point on a sphere so that $\frac{u_2v_2}{u_2v_1}$ is constant. This variation preserves $\frac{A(f)}{A(g)}$. Now suppose $c:=\frac{u_2v_2}{u_2v_1}$, then for small perturbations: 

\begin{equation}
\begin{split}
 \frac{u_2v_2}{u_2v_1}=c  \Leftrightarrow \\
 u_2(v_2-v_1c)=0.
 \end{split}
\end{equation}
So, in particular we can move $u_2$ along some geodesic so that  $\frac{A(f)}{A(g)}=1$. After some variation in $u_2$, we can pick $v_2$ and vary it similarly along geodesic preserving
\begin{equation}
\begin{split}
 \frac{u_2v_2}{u_3v_2}=c_2  \Leftrightarrow \\
 v_2(u_2-u_3c_2)=0.
 \end{split}
\end{equation}
Note that we can only ensure these variations locally, as we would like to preserve  the weak Aleksandrov relation, concentration condition, and since we want all other assignment functionals to remain less than $A(f)=A(g)$. All of these conditions are guaranteed by the openness and continuity of  the respective sets and functions. 

This illustration shows the geodesic variation of points on sphere so that solution to every variation doesn't exist. This variation preserves the edge-normal loop condition as well. 
\end{ex}

\section{Appendix}
These are results from  \cite{Semenov} which we use in the proof of Theorem \ref{main}. Recall that weak Aleksandrov relation was defined for discrete measures in Definition \ref{weak Aleksandrov related}. For general measures, we define the weak Aleksandrov relation as the following:

\begin{defi}\label{General Weak Alekss Defi}
Given two Borel measures $\mu$ and $\lambda$ on $\sn$, a measure $\mu$ is weak Aleksandrov related to $\lambda$ if $\mu(\sn)=\lambda(\sn)$ and for each closed set $\omega\subset\sn$ contained in closed hemisphere, there exists an $\alpha\in(0,1)$ such that
	\begin{equation}
		\mu(\omega)\leq\lambda(\omega_{\frac\pi2-\alpha}).
	\end{equation}
\end{defi}

Note that Proposition \ref{uniform constant} shows that Definition \ref{General Weak Alekss Defi} for general measures agrees with Definition \ref{weak Aleksandrov related} given for discrete measures. Equipped with more general form of weak Aleksandrov relation, we would like to mention the following results from \cite{Semenov} which we use in the proof of \ref{main}:

\begin{thm}\label{appendix main}[Theorem 1.4 in \cite{Semenov}]
	Suppose $\mu$ is a discrete Borel measure not concentrated on a closed hemisphere, and $\lambda$ is an absolutely continuous Borel measure. Suppose $\mu$ is weak Aleksandrov related to $\lambda$. Then, there exists a $K\in\kno$ such that $\mu=\lambda(K,\cdot)$. 
\end{thm}

\begin{lem}\label{weak Aleks bound}[Lemma 4.6 in \cite{Semenov}]
Suppose $\mu$ is a discrete Borel measure not concentrated on a closed hemisphere, and $\lambda$ is an absolutely continuous Borel measure. Suppose $\mu$ is weak Aleksandrov related to $\lambda$. Let $\alpha$ be a uniform constant for the weak Aleksandrov assumption. Then, there exists a polytope solution $P$ to the Gauss Image problem of the form: 
\begin{equation}
  P=(\bigcap_{i=1}^mH^-(\alpha_i,v_i))^*,
\end{equation}
such that $\frac{\mathfrak r_P}{\mathfrak R_p}$ is bounded from below by a constant $C_{\mu,\lambda}$, depending only on vectors $v_i$ and the uniform weak Aleksandrov constant $\alpha$. Besides being dependent on $\alpha$, this constant is independent of $\lambda$.
 \end{lem}

%%      ---------------------------------------------------------------------
%%      --------------------------- BIBLIOGRAPHY ----------------------------
%%      ---------------------------------------------------------------------

\frenchspacing
\bibliographystyle{cpam}

\begin{thebibliography}{99}



\bibitem{Aleks1}
A. Aleksandrov, An application of the theorem on the invariance of the domain to existense proofs, \textit{Izv. Akad. Nauk SSSR Ser. Math.} \textbf{3} no.3 (1939) 243--256.

\bibitem{Aleks0}
A. Aleksandrov, On the theory of Mixed Volumes. III. Extension of two theorems of Minkowski on convex polyhedra to arbitrary convex bodies. \textit{Mat. Sbornik N.S.} \textbf{3} (1938), 27-46.

\bibitem{Aleks}
A. Aleksandrov, Existence and uniqueness of convex surface with a given integral curvature. \textit{C. R. (Doklady) Acad. Sci. URSS (N.S.)} \textbf{35} (1942), 131-134.

\bibitem{Al1} S. Alesker, Continuous rotation invariant valuations on convex sets. \textit{Ann. of Math.} (2) 149
(1999), no. 3, 977--1005. 

\bibitem{Al2} 
S. Alesker, The multiplicative structure on continuous polynomial valuations. \textit{Geom. Funct.}
Anal. 14 (2004), no. 1, 1--26.

\bibitem{Al3}  S. Alesker, A. Bernig, F. E, Schuster, Harmonic analysis of translation invariant valuations.
\textit{Geom. Funct. Anal.} 21 (2011), no. 4, 751--773.

\bibitem{Kasia} 
S. Artstein-Avidan, S. Sadovsky and K. Wyczesany, A Rockafellar-type theorem for non-traditional costs, \textit{Adv. Math.} 395 (2022) 108157.


\bibitem{Caffarelli}
L. Caffarelli, Interior $W^{2,p}$-estimates for solutions of the Monge-Amp\'ere equation, \textit{Ann. Math} \textbf{131} (1990) 135-150.

\bibitem{L^p BM}
S. Chen, Y. Huang, Q.-R. Li, and J. Liu, The $L^p-$Brunn-Minkowski inequalities for $p<1$. \textit{Adv. Math} 368:107166 (2020).

\bibitem{Cheng}
S.-Y. Cheng and S.-T. Yau, On the regularity of the solution of the n-dimensional Minkowski problem, \textit{Commun. Pure Apple. Math.} \textbf{29} (1976) 495--516.

\bibitem{CW06adv}
K.-S.~Chou and X.-J.~Wang,
\textit{The {$L\sb p$}-Minkowski problem and the Minkowski problem
in centroaffine geometry},
Adv. Math.
\textbf{205} (2006), 33--83.


\bibitem{Bertrand} J. Bertrand, Prescription of Gauss curvature using optimal mass transport. \textit{Geom. Dedicata} \textbf{183} (2016), 81--99.

\bibitem{Bertrand2}
J. Bertrand, On the Gauss image problem, arXiv:2308.15810 (Aug 2023).

\bibitem{Besau} F. Besau and E. M. Werner, The floating body in real space forms. \textit{J. Differential Geom.} \textbf{110} (2008), no.2, 187--220.



\bibitem{BHP17jdg}
K. J. B\"or\"oczky, M. Henk, and H. Pollehn,
\textit{Subspace concentration of dual curvature measures of symmetric convex bodies},
J. Differential Geom.
\textbf{109} (2018), 411--429.


\bibitem{BLYZ13jams}
K. J. B\"or\"oczky, E. Lutwak, D. Yang, and G. Zhang,
\textit{The logarithmic Minkowski problem},
J. Amer. Math. Soc. (JAMS)
\textbf{26} (2013), 831--852.




\bibitem{log}
K. J. B\"or\"oczky, E. Lutwak, D. Yang and G. Zhang, The log-Brunn-Minkowski inequality, \textit{Adv. Math.} 231 (2012) 1974-1997.

\bibitem{GIP}
K. J. B\"or\"oczky, E. Lutwak, D. Yang, G. Zhang and Y. M. Zhao, The Gauss Image problem, \textit{Commun. Pure Apple. Math.} \textbf{Vol.LXXIII} (2020), 1046--1452. 


\bibitem{Brualdi}
R. A. Brualdi, 
\textit{Combinatorial Matrix Classes}, 
Encyclopedia of Mathematics and its Applications, Cambridge University Press, Cambridge, 2006.

 
\bibitem{Firey}
W. J. Firey,
Mean cross-section measures of harmonic means of convex bodies,
\textit{Pacific J. Math.} \textbf{11} (1961), 1263--1266.

\bibitem{Firey2}
W. J. Firey,
p-Means of Convex Bodies, \textit{Math. Scand.} \textbf{10} (1962) 17--24.

\bibitem{G06book}
R.J.~Gardner,
\textit{Geometric Tomography},
Second edition, Encyclopedia of Mathematics and its Applications,
Cambridge University Press, Cambridge,
 2006.
 
 \bibitem{Gruberbook}
P.M.~Gruber,
\textit{Convex and discrete geometry},
Grundlehren der Mathematischen Wissenschaften, 336, Springer,
Berlin, 2007.


\bibitem{HaberlSchuster}
C. Haberl and F. E. Schuster,
 Asymmetric affine Lp Sobolev inequalities. \textit{J. Funct. Anal.}, 257:641--658, (2009).


\bibitem{Ramon}
R. Van Handel, The local logarithmic Brunn-Minkowski inequality for zonoids, \textit{Geom. Aspects of Funct. Anal.,} Lecture Notes in Mathematics 2327, Springer, pp. 355-379 (2023).

\bibitem{Lp Aleks}
Y. Huang, E. Lutwak, D. Yang and G. Zhang,
The $L_p$-Aleksandrov problem for $L_p$-Integral curvature,
\textit{J. Differential Geometry} \textbf{110} (2018), no. 1, 1--29.


\bibitem{HLYZ16}
Y. Huang, E. Lutwak, D. Yang, and G. Zhang,
Geometric measures in the Brunn-Minkowski theory and their associated Minkowski problems, \textit{Acta Math.} \textbf{216} (2016), 325--388.

\bibitem{HZ18adv}
Y. Huang and Y. Zhao,
\textit{On the $L_p$ dual Minkowski problem},
Adv. Math. \textbf{332} (2018), 57--84.

\bibitem{GaussProb}
Y. Huang, D. Xi, and Y. Zhao. The Minkowski problem in Gaussian probability space. \textbf{Adv. Math.,} 385:Paper No. 107769, 36, 2021.


\bibitem{KolesnikovMilman}
A. V. Kolesnikov and E. Milman, Local $L^p$-Brunn-Minkowski inequalities for $p<1$, \textit{Mem. Amer. Math. Soc.} \textbf{277} (2022)

\bibitem{L1}
M. Ludwig, Ellipsoids and matrix-valued valuations., \textit{Duke Math. J.} \textbf{119} (2003), no. 1, 159--188.

\bibitem{L2}
M. Ludwig, Intersection bodies and valuations. \textit{Amer. J. Math.} \textbf{128} (2006), no. 6, 1409--1428.

\bibitem{L3}
M. Ludwig, Minkowski areas and valuations. \textit{J. Differential Geom.} \textbf{86} (2010), no. 1, 133--161.

\bibitem{L4}
M. Ludwig; M. Reitzner, A classification of SL.n invariant valuations. \textit{Ann. of Math.} \textbf{2} (2010), no. 2, 1219--1267
 
\bibitem{LutwakLp} 
E. Lutwak, The Brunn-Minkowski-Firey Theory. I. Mixed volumes and the Minkowski problem,  \textit{J. Differential Geometry}  \textbf{38} (1993) 131--150.
 
\bibitem{Lutwak}  
E. Lutwak, Dual mixed volumes, \textit{Pacific J. Math.}, \textbf{58} (1975), 531-538.



\bibitem{LO95jdg}
E.~Lutwak and V.~Oliker,
\textit{On the regularity of solutions to a generalization of the Minkowski problem},
J. Differential Geom. \textbf{41}
 (1995), 227--246.

\bibitem{LYZ00jdg}
E.~Lutwak, D.~Yang, and G.~Zhang,
\textit{${L}\sb p$ affine isoperimetric inequalities},
J. Differential Geom. \textbf{56} (2000), 111--132.

\bibitem{MR1987375}
E.~Lutwak, D.~Yang, and G.~Zhang,
\textit{Sharp affine {$L_p$} {S}obolev inequalities},
J. Differential Geom. \textbf{62} (2002), 17--38



\bibitem{LYZ04tams}
E.~Lutwak, D.~Yang, and G.~Zhang,
\textit{On the $L\sb p$-Minkowski problem},
Trans. Amer. Math. Soc. \textbf{356} (2004),  no. 11, 4359--4370.


\bibitem{LYZ06imrn}
E.~Lutwak, D.~Yang, and G.~Zhang,
\textit{Optimal Sobolev norms and the $L\sp p$ Minkowski problem},
 Int. Math. Res. Not.
2006, Art. ID 62987, 21.


\bibitem{LYZ16}
E.~Lutwak, D.~Yang, and G.~Zhang,
\textit{$L_p$ dual curvature measures},
Adv. Math. \textbf{329} (2018), 85--132.

\bibitem{Milman}
E. Milman. Centro--affine differential geometry and the log--Minkowski problem, to appear in J. Eur. Math. Soc., arXiv:2104.12408

\bibitem{Stephanie}
S. Mui, On the $L^p$ Aleksandrov Problem for negative p, \textit{Adv. math}, 408 (2022). 

\bibitem{Nirenberg}
L. Nirenberg, The Weyl and Minkowski Problems in differential geometry in the Large, \textit{Commun. Pure Appl. Math.}, \textbf{6} (1953), 337--394.

\bibitem{Ol2}
V. Oliker,
\textit{Existence and uniqueness of convex hypersurfaces with prescribed
Gaussian curvature in spaces of constant curvature},
Sem. Inst. Matem. Appl. Giovanni Sansone (1983), 1--64.

\bibitem{Ol21}
V. Oliker,
\textit{Hypersurfaces in $\mathbb R^{n+1}$ with prescribed Gaussian curvature and related equations of Monge-Amp\` ere type}, 
Comm. Partial Differential Equations
\textbf{9} (1984), 807--838.

\bibitem{Oliker}
V. Oliker, Embedding $\sn$ into $\mathbb R^{n+1}$ with given integral Gauss curvature and optimal mass transport on $\sn$. \textit{Adv. Math} \textbf{213} no. 2 (2007) 600-620.

\bibitem{Pogorelov}
A.V. Pogorelov, The Minkowski Multidimensional Prroblem, \textit{V.H. Winston and Sons} (1978) Washington, D.C. 

\bibitem{Rochet} 
J.-Ch. Rochet, A necessary and sufficient condition for rationalizability in a quasi-linear context, \textit{Journal of mathematical Economics} \textbf{16} (1987), no. 2, 191-200.



 \bibitem{Rock}
R. T. Rockafellar,
\textit{Convex Analysis}, 
Princeton University Press, Princeton, Second Printing, 1972.

\bibitem{cycl}
R. T. Rockafellar, 
Characterizations of the subdiferentials of convex functions, \textit{Pacific J. Math.} \textbf{17} (1966), 497--510.



\bibitem{rush}
L. R\"{u}schendorf, On c-optimal random variables, \textit{Statistics \& probability letters} \textbf{27} (1996), no.3, 267-270.

\bibitem{Saroglou}
C. Saroglou, Remarks on the conjectured log-Brunn-Minkowski inequality. \textit{Geom. Dedicata}, \textbf{177} (2015), 353-365.

\bibitem{S14}
R.~Schneider,
\textit{Convex Bodies: The Brunn-Minkowski Theory},
 Second Edition, Encyclopedia of Mathematics and its Applications, Cambridge
 University Press, Cambridge (2014).
 
\bibitem{Sh1} 
 F. E. Schuster. Valuations and Busemann-Petty type problems. \textit{Adv. Math.},
219:344--368, 2008.

\bibitem{Sh2} 
F. E. Schuster. Crofton measures and Minkowski valuations. \textit{Duke Math. J.},
154:1--30, 2010.

\bibitem{Sh3} 
F. E. Schuster and T. Wannerer. GL(n) contravariant Minkowski valuations.
\textit{Trans. Amer. Math. Soc.}, 364:815--826, 2012.

\bibitem{Semenov}
V. Semenov, The Gauss Image Problem with weak Alexandrov condition, \textit{J. Funct. Anal.}, \textbf{287} (2024)

\bibitem{Sta1}
A. Stancu,
\textit{The discrete planar $L_0$-Minkowski problem},
Adv. Math. \textbf{167} (2002), 160--174.


\bibitem{Trudinger} N. S. Trudinger and X.-J. Wang, The Monge-Amp\`ere equation and its geometric applications. \textit{Handb. Geom. Anal.} (2008), 467--524.

\bibitem{Lp Gauss} C. Wu, D. Wu, and N. Xiang, The $L^p$ Gauss Image Problem, \textit{Geometriae Dedicata}, Volume 216, article number 62, (2022). 

\bibitem{Wyczesany} K. Wyczesany, Topics in high-dimensional geometry and optimal transport (Doctoral thesis) (2020).

\bibitem{Dongmeng}
D. Xi., G. Leng., Dar's conjecture and the log--Brunn--Minkowski inequality. \textit{J. Differential Geom.} \textbf{103} (1) 145--189, May 2016.

\bibitem{YZCVPDE}
Y. Zhao,
\textit{The dual Minkowski problem for negative indices},
Calc. Var. Partial Differential Equations 56(2):Art. 18, 16,
  2017.

\bibitem{YZJDG}
Y. Zhao,
\textit{Existence of solutions to the even dual {M}inkowski problem},
\textit{J. Differential Geom.} \textbf{110} (2018), 543--572.



\bibitem{Zhao}
Y. M. Zhao, The $L^p$ Aleksandrov problem for origin-symmetric polytopes, \textit{Proc. Amer. Math. Soc.}  \textbf{147} (2019), 4477-4492. 



\bibitem{Zu}
G. Zhu,
\textit{The logarithmic Minkowski problem for polytopes},
Adv. Math. \textbf{262} (2014), 909--931.

\bibitem{Zu2}
G. Zhu,
\textit{The centro-affine Minkowski problem for polytopes},
J. Differential Geom. \textbf{101} (2015), 159--174.


\end{thebibliography}

\end{document}